\documentclass[10pt]{extarticle}
\usepackage[utf8]{inputenc}
\usepackage{amsfonts}
\usepackage{amsthm}
\usepackage[fleqn]{amsmath}
\usepackage{mathrsfs} 
\usepackage{ amssymb }
\usepackage[margin=1in,footskip=0.3in]{geometry}
\usepackage{hyperref}  
\usepackage{cleveref}
\usepackage{graphicx}
\usepackage{tikz}
\usetikzlibrary{cd}
\usepackage[all]{xy}
\usetikzlibrary{matrix,arrows,positioning, decorations.pathmorphing,backgrounds,fit,petri,calc}
\usepackage{tkz-euclide}
\usepackage{float}
\newtheorem{theorem}{Theorem}[section]
\newtheorem{proposition}[theorem]{Proposition}
\newtheorem{definition}[theorem]{Definition}
\newtheorem{corollary}[theorem]{Corollary}
\newtheorem{conjecture}[theorem]{Conjecture}
\newtheorem*{theorem*}{Theorem}

\usepackage{setspace}
\usepackage{parskip}
\renewcommand{\indent}{\hspace*{20pt}}

\usepackage[style=alphabetic,citestyle=alphabetic]{biblatex}
\usepackage{xcolor}

\usepackage[normalem]{ulem}

\newtheorem{remark}[theorem]{Remark}
\newtheorem{lemma}[theorem]{Lemma}
\newtheorem*{proposition*}{Proposition}

\newcommand{\Z}{\mathbb Z}
\newcommand{\SL}{\mathsf{SL}}

\newcommand{\C}{\mathbb C}
\newcommand{\Q}{\mathbb Q}
\DeclareMathOperator{\Hom}{Hom}
\renewcommand{\P}{\mathbb P}

\newcommand{\NN}{\widetilde{\mathcal N}}
\newcommand{\N}{\mathcal N}

\newcommand{\g}{\mathfrak{g}}
\newcommand{\pr}{\mathrm{pr}}
\renewcommand{\b}{\mathfrak{b}}
\newcommand{\p}{\mathfrak{p}}
\newcommand{\h}{\mathfrak h}
\newcommand{\z}{\mathbf{z}}

\renewcommand{\O}{\mathcal O}
\newcommand{\n}{\mathfrak{n}}
\renewcommand{\u}{\mathfrak{u}}
\renewcommand{\sl}{\mathfrak{sl}}

\DeclareMathOperator{\ad}{ad}

\DeclareMathOperator{\HH}{HH}
\renewcommand{\H}{\operatorname{H}}
\DeclareMathOperator{\Sym}{Sym}
\DeclareMathOperator{\tensor}{\otimes}
\DeclareMathOperator{\dirsum}{\oplus}
\DeclareMathOperator{\Id}{Id}
\DeclareMathOperator{\im}{im}
\DeclareMathOperator{\BGG}{BGG}
\DeclareMathOperator{\Ind}{Ind}

\title{On certain Hochschild cohomology groups for the small quantum group}

\author{Nicolas Hemelsoet\thanks{University of Geneva}\and Rik Voorhaar\footnotemark[1]}
\date{}

\begin{document}

\maketitle 

\begin{abstract}

We apply the sheaf cohomology BGG method developed by the authors and Lachowska-Qi to the computation of Hochschild cohomology groups of various blocks of the small quantum group. All our computations of the center of the corresponding block agree with the conjectures of Lachowska-Qi. In the case of the nontrivial singular block for $\g = \sl_3$, we obtain the $\H^{\bullet}(\u, \C) = \C[\N]$-module structure of $\HH^{\bullet}(\u_{\lambda})$. 

%We compute several Hochschild cohomology groups of blocks of small quantum groups. Our first result is the computation of the full Hochschild cohomology ring of the non-trivial singular block for $\sl_3$ as $G$-module. We present an alternative description in terms of the nilpotent cone. We obtain partial results for projective spaces. Finally, using the algorithm introduced by the authors in a previous work, we compute various tables of Hochschild cohomology of blocks of small the quantum group. In particular, we obtain the center of all the blocks of the small quantum group for type $B_3,C_3$ and $A_4$, and confirm several conjectures of Lachowska-Qi for these cases. 
\end{abstract}

\tableofcontents

\section{Introduction}\label{sec:intro}

Given a semisimple complex Lie algebra $\mathfrak g$ and a positive integer $\ell$, Lustzig defined in \cite{L2} a remarkable finite-dimensional algebra, \emph{the small quantum group}. It is an open question to compute its center. The center of the small quantum group is an important object of study, and it turns out to be related to other areas of representation theory. Significant progress was made in \cite{BL}, where the center was identified with certain sheaf cohomology groups over the Springer resolution. Its dimension has been computed for $\g = \sl_2$ in \cite{K} , for $\g = \sl_3$ in \cite{LQ}, for $\g=\sl_4$ and $\g = \b_2$ in \cite{LQ2}. 

\indent In \cite{LQ}, the \emph{sheaf-cohomology BGG algorithm} was introduced, and used to compute the center for $\g = \sl_3$. However, the computations for higher ranks are too complicated to be done by hand, even for type $\mathfrak g_2$. Later in \cite{HV}, we developed the method further and implemented it as a software package. In our previous work, we have provided several geometric applications of the method, but the original motivation was to study the center of the small quantum group. This is what we will do in this paper. 

\indent In section \ref{sec:SmallQuantumGroup}, we recall some facts about the small quantum group. In section $3$ we compute the full ring $\HH^{\bullet}(\u_1)$, where $\u_1$ is the non-trivial singular block for $\g = \sl_3$. We express this ring as $\g$-module, and as module over the functions on the nilpotent cone $\C[\N]$, using the sheaf cohomology BGG method. The $\C[\N]$-module of $\HH^{\bullet}(\u_0)$ was computed in \cite{LQ3}, for $\g = \sl_2$. In section $4$, we compute the center of all blocks of $G_2,B_3,C_3$ and $A_4$. We confirm the conjectures by Lachowska-Qi in each case. In section $5$, we discuss the higher Hochschild cohomology groups and give several examples. In section $6$, we consider the case when $G/P \cong \P^n$. In this case, $\HH^0$ was computed in \cite{LQ2}. We present partial results for higher Hochschild cohomology. 

\subsection*{Acknowlegements}

We are grateful to Anna Lachowska who suggested the project and explained her work to us. The first author would like to thank Qi You for useful discussions. This publication was produced within the scope of the NCCR SwissMAP which was funded by the Swiss National Science Foundation. The authors would like to thank the Swiss National Science Foundation for their financial support.

\section{The small quantum group and the BGG complex}\label{sec:SmallQuantumGroup}

In this section, we recall the definition of the small quantum group and some of its properties. We also recall the algorithm of Lachowska-Qi to compute its center. 
Then, we will give explicit formulas of the BGG maps for higher cohomology $\sl_3$, and give a simplified description of the module $V_{j,k}$ which are used to compute the Hochschild cohomology of blocks of the small quantum group.

\subsection{The center of the small quantum group }

As in the introduction, let $\g$ be a semisimple complex Lie algebra. Let $\ell$ be an odd integer number coprime to the index of connection of $\g$, greater than the Coxeter number of $\g$ and coprime to $3$ if $\g$ contains a factor of type $G_2$. Let $U_v$ be the Drinfeld-Jimbo quantum group defined over $\Q(v)$, and let $\mathcal A = \Z[v,v^{-1}]$. Lusztig introduced a $\mathcal A$-form of $U_v$, denoted by $U_{\mathcal A}$. It is generated by the divided powers $E_i^{(r)}, F_i^{(r)}$ and the elements $\binom{K_{\mu},m}{n}$, see \cite{L1} for more details. After localisation at a primitive $\ell$-root of unity,  this becomes the big quantum group $\rm{U}$.

Lusztig introduced a finite-dimensional analogue of the first Frobenius kernel of an algebraic group in positive characteristic: 

\begin{definition}[\cite{L2}]
The small quantum group, written $\u_q(\g)$ is the subalgebra of $\mathrm {U}$ generated by $E_i,F_i$ and $K_i, K_i^{-1}$.
\end{definition}

\begin{proposition}\cite{LQ3}
There is a $\g$-action on the center $\z(\u_q(\g))$.
\end{proposition}

Hence, it is natural to try to understand the center not only as vector space but also as $\g$-module. Let $\theta^{\vee}$ be the highest coroot, and $\overline{\mathcal C} = \{ \lambda \in P^+: 0 \leq (\lambda + \rho, \theta^{\vee}) \leq \ell \}$.

\begin{proposition}
There is a decomposition (as two-sided ideals)  
\[\u_q(\g) \cong \bigoplus_{\lambda \in \mathscr S} \u_\lambda(\g) \]where $\mathscr S$ is the set of orbits in $\overline {\mathcal C}$ under the action of the $\ell$-extended affine Weyl group. 
\end{proposition}

Hence, there is a corresponding decomposition of the center: $\z(\u_q(\g)) \cong \bigoplus_{\lambda} \z(\u_{\lambda}(\g))$. These summands are called \emph{blocks}, and the block with $\lambda = 0$ is called the \emph{principal block}. A block not equivalent to the principal block is called a \emph{singular block}. 

\begin{definition}
Fix a parabolic subgroup $P \subset G$, where $G$ is the connected semsimple Lie group of adjoint type associated to $\g$. The \emph{Springer resolution} is $\NN_P:= T^*(G/P)$. We write $\pr$ for the projection $\pr: T^*(G/P) \to G/P$.
\end{definition}

Let $\n_P$ be the nilradical of $\p$, where $\p = \rm{Lie}(P)$. We can identify $T^*(G/P) \cong G \times^P \n_P$, and if $\N \subset \g$ is the nilpotent cone, there is a map $\mu: T^*(G/P) \to \N$. If $P = B$ is a Borel subgroup, this map is a resolution of singularities. 

We will consider $\C^*$-equivariant sheaves on $\NN_P$ for the action given by $(t,v) \mapsto (t^{-2}v)$, where $v$ is the fiber coordinate.

\begin{theorem}[\cite{BL}]
There is an isomorphism  $\HH^s_{\C^{\times}}(\NN) \cong \HH^s(\u_0) $ where the left-hand side is the $\C^*$-equivariant Hochschild cohomology of $\NN$.
\end{theorem}

This isomorphism is compatible with the $\g$-module structure (\cite{LQ3}).  In \cite{LQ} and \cite{LQ2}, the left-hand side was explicitely computed in terms of the BGG resolution associated to a finite-dimensional irreducible representation of $\g$. Let us recall briefly how to do it.  There are certain $G$-equivariant vector bundles $\mathcal V_{j,k} = G \times^B V_{j,k}$ (where $j,k$ are integers) on $G/B$ such that the Hochschild cohomology can be obtained from sheaf cohomology of the $\mathcal V_{j,k}$. We will describe the $B$-modules $V_{j,k}$ in section~\ref{sec:Vjk}. The precise relation is \begin{equation}\label{eq:hochschildSpringer}
    \HH_{\C^*}^{s}(\NN) \cong \bigoplus_{i+j+k = s} \H^i(G/B, \mathcal V_{j,k}) 
\end{equation}

It follows that the center has a natural bigrading. The authors of \cite{LQ} noticed that one can reduce the computation of the right-hand side of equation~\eqref{eq:hochschildSpringer} to a Lie algebra cohomology computation. A convenient tool for it was the \emph{BGG resolution} (we will recall its basic properties in the next subsection). These observations from \cite{LQ} lead to the following structure result: 

\begin{theorem}[\cite{LQ}, Theorem 4.3]
For any $s \geq 0$, there is an $\sl_2$-action on the Hochschild cohomology $\HH^{s}(\u_0)$, where the generator $e \in \sl_2$ acts as a homogeneous element of bidegree $(i,j) = (0,2)$. 
\end{theorem}

More precisely, this homogenous element is the Poisson bivector field $\tau \in H^0(\NN, \wedge^2 T \NN)^{-2}$ (which is dual of the canonical symplectic form $\omega \in H^2(\NN, \Omega^2)^2$). Hence the action of $e$ is given by the map \[ \tau \wedge - \colon \H^i(\NN, \wedge^j T \NN)^k \to \H^i(\NN, \wedge^{j+2}T\NN)^{k-2} \] 
In particular, there are canonical isomorphisms (\cite{LQ}, Corollary 4.4) \[ \tau^j \wedge - \colon \H^i(\NN, \wedge^{n-j} T \NN)^k \to \H^i(\NN, \wedge^{n+j} T \NN)^{k-2j} \]
We can represent the action of $\tau$ on the bigraded table as follows (we took $\g = \sl_3$ and $s=0$): 

 \begin{tikzcd}
   { i+j=0} \arrow[dddd, dash, start anchor={[xshift=6.5ex,yshift=3ex]}, end anchor={[xshift=6.5ex,yshift=-3ex]}] &[-4ex] \H^0(\NN, \wedge^0 T \NN)^0 \ar[rd,dashed, "\tau \wedge -"] & & &\\
	{ i+j=2} & \H^1(\NN, \wedge^1 T \NN)^{-2} \ar[rd,dashed, "\tau \wedge -"] & \H^0(\NN, \wedge^2 T \NN)^{-2} \ar[rd,dashed, "\tau \wedge -"] & & \\
	{ i+j=4} & \H^2(\NN, \wedge^2 T \NN)^{-4} \ar[rd,dashed, "\tau \wedge -"] & \H^1(\NN, \wedge^2 T \NN)^{-4} \ar[rd,dashed, "\tau \wedge -"] & \H^0(\NN, \wedge^4 T\NN)^{-4} \ar[rd,dashed, "\tau \wedge -"] & \\
	{ i+j=6} & \H^3(\NN, \wedge^3 T \NN)^{-6} & \H^{2}(\NN, \wedge^4 T \NN)^{-6} & \H^{1}(\NN, \wedge^5 T \NN)^{-6} &  \H^0(\NN, \wedge^6 T \NN)^{-6}  \\[-1.5em]  h^{i,j} \arrow[rrrr,dash, start anchor={[yshift=2.5ex, xshift=-3.5em]}, end anchor={[yshift=2.5ex, xshift=6em]}]&{ j-i=0}&{ j-i=2}&{ j-i=4}&{ j-i=6}
 \end{tikzcd} 

Hence, to compute the bigraded table corresponding to $\HH^s_{\C^*}(\NN)$, we just need to compute the top-left half of the table (including the diagonal starting at the bottom left), and use the $\sl_2$-action to obtain the full table. \\

In \cite{LQ}, the center of $\u_0(\sl_3)$ was obtained as bigraded vector space , and the authors noticed that as a bigraded vector space, $\HH^0(\sl_3)$ was isomorphic to the double coinvariant algebra ${\rm{DC}}_3$, where \[ {\rm{DC}}_m = \C[x_1, \dots, x_m, y_1, \dots, y_m]/I, \] and $I$ is the ideal generated by invariant polynomials (for the diagonal $\mathfrak S_m$-action). This motivated the following conjecture: 

\begin{conjecture}[\cite{LQ}]
As a bigraded vector space, there is an isomorphism $\z(\u_0(\sl_m)) \cong \mathrm{DC}_m$. In particular, $\dim \z(\u_0(\sl_m)) = (m+1)^{m-1}$.
\end{conjecture}

It is also conjectured that $\z(\u_0(\sl_m))$ only contains the trivial representation. In other types, the presence of non-trivial representations made the formulation of a general conjecture more difficult. Computations for $\b_2$ showed that the $\g$-invariant part was also isomorphic as bigraded vector space to the double-coinvariant algebra (which is defined for any Weyl group $W$). Hence one could still hope for the following: 

\begin{conjecture}\cite{LQ2}\label{conj:diag}
Let $\g$ be a semisimple Lie algebra with Weyl group $W$. Then, $\z(\u_0(\g))^\g \cong \mathrm{DC}(W)$.
\end{conjecture}

For $\g = \sl_n$, it is expected that the whole center is $\g$-invariant:

\begin{conjecture}\cite{LQ}\label{conj:triv}
Let $\g = \sl_n$. Then, $\z(\u_q(\sl_n)) = \z(\u_q(\sl_n))^{\sl_n}$.
\end{conjecture}

Finally, outside the simply-laced case, non-trivial representations appear. Based on our computations, we make the following conjecture: 

\begin{conjecture}
Let $L$ be an irreducible non-trivial representation appearing in $z(\u_q(\g))$, and $h$ be the Coxeter number of $\g$. Then, $h+1$ divides $\dim L$. 
\end{conjecture}

\subsection{The BGG complex}

Fix a triangular decomposition $\g = \u \dirsum \h \dirsum \n$, and let $\b = \h \dirsum \n$. For any $\mu \in \h^*$, recall that the Verma module is defined as $M(\mu) = \Ind_{U(\h \dirsum \n)}^{U(\g)} \C_{\mu}$, where $\C_{\mu}$ is a one dimensional $\b$-module, where $\h$ acts by $\mu$ and $\n$ by zero. Let $P^+$ be the set of dominant weights, and $\lambda \in P^+$. 

\begin{theorem}[see \cite{humphreys}, chapter $6$]
There is an exact sequence \[ 0 \to M(w_0 \cdot \lambda) \to \bigoplus_{\ell(w) = n-1} M(w \cdot \lambda) \to \dots \bigoplus_{\ell(w) = k} M(w \cdot \lambda) \dots \to M(\lambda) \to L(\lambda) \to 0 \]
\end{theorem}

Here, $w \cdot \mu = w(\mu + \rho) - \rho$ is the dot-action, where $\rho$ is half the sum of positive roots. We call $\BGG^{\bullet}(\lambda)$ the resulting complex made of Verma modules. The maps are given by certain monomials in $U(\n)$. If $E$ is a finite-dimensional $\b$-module, there is an associated complex $\BGG^{\bullet}(E, \lambda):= (\Hom_{\n}(\BGG^{\bullet}(\lambda), E))^{\h}$. Using that Verma modules are free $U(\n)$-modules, we can identify \[ \BGG^k(E, \lambda) =  \bigoplus_{\ell(w) = k} E[w \cdot \lambda ],\]
here for a weight $\mu$, $E[\mu]$ is the corresponding weigh space. This complex can be used to compute certain sheaf cohomology groups: 

\begin{proposition}[\cite{LQ}]
Let $\mathcal E = G \times^B E$ the associated vector bundle on $G/B$. Then, \[ \Hom_G(L(\lambda), \H^{\bullet}(G/B, \mathcal E)) \cong \H^{\bullet}(\BGG(E, \lambda)). \]
\end{proposition}

For example, if $\g = \sl_3$, and  $\lambda = n_1 \alpha_1 + n_2 \alpha_2$ is a dominant root weight, i.e. $n_1,n_2 \geq 0$, $2n_2 \geq n_1$ and $2n_1 \geq n_2$, the corresponding BGG complex is 

\newcommand{\vertEspace}[2]{E\left[\begin{matrix}#1\\ #2\end{matrix}\right]}

\[
    \begin{tikzcd}[column sep=8ex, every label/.append
style={description, font=\normalsize}]
        &  \vertEspace{-n_1+n_2-1}{n_2}  \ar[r,"s_1s_2s_1"] \ar[rdd,"s_2" near end]&[4em] \ar[rd,"s_1"] \vertEspace{-n_2-2}{n_1-n_2 -1} &  &   \\
    \vertEspace{n_1}{n_2} \ar[rd,"s_1"] \ar[ru,"s_2"] &  &  & \vertEspace{-n_2-2}{-n_1-2} \\
      & \vertEspace{n_1}{n_1-n_2-1}  \ar[r, "s_2s_1s_2"'] \ar[ruu,"s_1" near end] & \vertEspace{-n_1+n_2 -1}{ - n_1-2}  \ar[ru,"s_2"] & &
    \end{tikzcd}
\] 

Here $\vertEspace{a_1}{a_2}$ denotes the weight space of weight $a_1 \alpha_1 + a_2 \alpha_2$. 

% Backup of old code.
% \[
%     \begin{tikzcd}
%         &  -(n_2+2)\alpha_1 + (n_1-n_2-1)\alpha_2 \ar[r,"s_1s_2s_1"] \ar[rdd,"s_2"]&  (-n_1+n_2-1)\alpha_1 + n_2 \alpha_2 \ar[rd,"s_1"]  &  &   \\
%     -(n_2+2)\alpha_1 - (n_1+2)\alpha_2   \ar[rd,"s_1"] \ar[ru,"s_2"] &  &  & \lambda  \\
%         & (-n_1+n_2 -1)\alpha_1 - (n_1+2)\alpha_2 \ar[r, "s_2s_1s_2"'] \ar[ruu,"s_1"] & n_1\alpha_1 + (n_1-n_2-1) \alpha_2 \ar[ru,"s_2"] & &
%     \end{tikzcd}
% \] 

This for example implies that the map $E[\lambda] \to E[s_1 \cdot \lambda]$ is given by multiplication by $f_1^{2n_1-n_2+1}$, since it has to be $\h$-equivariant. Up to symmetry, the only non-trivial map is the map $E[s_1 \cdot \lambda] \to E[s_1 s_2 \cdot \lambda]$. Let us compute this map, even though we will not need it explicitly. This map is given by multiplication of an element $f_{1 \to 12} \in U(\n)$ of weight $(n_2+1)\alpha_1 + (2n_2 - n_1+1)\alpha_2$, and satisfies the equation \[ f_{1 \to 12} f_1^{2n_1 - n_2 +1} = f_1^{n_1 + n_2 + 2} f_2^{-n_1+2n_2 +1} \] 

Assume $m \geq n$. Then, we have \[ f_1^mf_2^n = \sum_{0 \leq r \leq n} \binom{n}{r} \left( \prod_{j=0}^{r-1} (m-j) \right) f_{12}^rf_2^{n-r}f_1^{m-r} \]

It follows that $f_1^mf_2^n$ is right divisible by $f_1^a$ if and only if $m \geq a + n$. This is the case for the operator $f_{1 \to 12}$ in our BGG complex, hence we obtain:

\begin{corollary}
The BGG operator introduced before is given by \setlength{\mathindent}{5pt}
\[
f_{1 \to 12} = \sum_{0 \leq r \leq -n_1+2n_2 +1} \binom{-n_1+2n_2 +1}{r} \left( \prod_{j=0}^{r-1} (n_1 + n_2 + 2-j) \right) f_{12}^rf_2^{-n_1+2n_2 +1-r}f_1^{-n_1 + 2n_2 + 1-r} 
\]
\end{corollary}

We will use the BGG complex in the next section to compute Hochschild cohomology of the block $\u_1(\sl_3)$. Let us mention that a parabolic version of the BGG resolution exists (cf. \cite[\textsection 9.16]{humphreys}), but implementing a parabolic version of the sheaf-cohomology BGG algorithm is not easy. However, since we are only interested in computing sheaf cohomology, we can simply consider a $\p$-module $E$ as a $b$-module by restriction and compute its cohomology as $\b$-module. This gives the right answer because the spectral sequence associated to the projection $G/B \to G/P$ degenerates.

\subsection{The modules \texorpdfstring{$V_{j,k}$}{Vjk}}\label{sec:Vjk}

To compute $\HH_{\C^*}^{s}(\NN)$ using $\eqref{eq:hochschildSpringer}$ we need a description of the modules $V_{j,k}$. Let us recall the construction of $V_{j,k}$ from \cite{LQ} and \cite{LQ2}. Let $P$ be a standard parabolic subgroup, $\p = \text{Lie}(P)$, $\pr: \NN_P = T^*G/P \to G/P$ be the projection, $\n_{\p}$ the nilradical of $P$, and $\u_{\p}$ its dual seen as a subset of $\g$ using the Killing form. Recall that $\p$ acts on its nilradical by the adjoint action, and on $\u_{\p}$ by the coadjoint action. Hence we have a map $\text{ad}: \p \to \text{End}(\n_{\p}) \cong \n_{\p} \tensor \u_{\p}$. Let $\Delta = (\iota, \text{ad}): \p \to \g \dirsum \n_{\p} \tensor \u_{\p}$, where $\iota$ is the inclusion. We also write $\Delta$ for the induced map of $\Sym^{\bullet}(\u_{\p})$-modules 
\begin{equation}
\Delta: \Sym^{\bullet}(\u_{\p}) \tensor \p \to \Sym^{\bullet}(\u_{\p})  \tensor (\g \dirsum \n_{\p} \tensor \u_{\p}).    
\end{equation}
 Here, the $\p$-action on $\Sym^{\bullet}(\u_{\p})$ is induced from the coadjoint action, and the action on $\g$ and $\n_{\p}$ is the adjoint action.

\begin{proposition}\cite{LQ2}
Let $\pr_* T_{\NN_P} = G \times^P V_1$. Then, we have an isomorphism of $\p$-modules

\begin{equation}\label{eq:V1def}
V_1 \cong \frac{\Sym^{\bullet}(\u_{\p}) \tensor \g \dirsum \Sym^{\bullet}(\u_{\p}) \tensor \n_{\p}}{ \Sym^{\bullet}(\u_{\p}) \tensor \Delta(\p)}
\end{equation}
\end{proposition}

All exterior powers of $\pr_* T_{\NN_P}$ are described in \cite{LQ2} as well. For example, if $V_2:=\wedge^2_{\Sym^{\bullet}(\u_{\p})} V_1 $ then \[V_2 \cong \frac{\Sym^{\bullet}(\u_{\p}) \tensor (\wedge^2 \g \dirsum \g \tensor \n_{\p} \tensor \wedge^2 \n_{\p})}{\Delta(\Sym^{\bullet}(\u_{\p}) \tensor \p) \wedge ( \Sym^{\bullet}(\u_{\p}) \tensor (\g \dirsum \n_{\p}))}\] 

Now, we recall that for application to the small quantum group, we need a $\C^*$-grading that we will call the $k$-grading. 

\begin{definition}
The \emph{k-grading} is the following grading on $\Sym^{\bullet}(\u_{\p}) \tensor \g \dirsum \Sym^{\bullet}(\u_{\p}) \tensor \n_{\p}$ : $\deg(\n_{p}) = -2$, $\deg(\Sym^m(\u_{\p})) = 2m$ and $\deg(\g) = 0$. We write $V_{j,k}$ for the $k$-th graded part of $\wedge^j_{\Sym^{\bullet}(\u_{\p})} (V_1)$ (for the natural induced grading). Let us notice that $k$ has to be even. 
\end{definition}

We can give an alternative description of the $\p$-modules $\wedge^j_{\Sym^{\bullet}(\u_{\p})}V_1$ and the graded pieces $V_{j,k}$, that is more convenient algorithmically.  These results essentially follow from the short exact sequence of vector bundles introduced in \cite{LQ}: (recall that $\Omega_X$ is the cotangent sheaf of $X$)
\[ 0 \to \pr_* \O_{\NN_P} \tensor \Omega_{G/P} \to \pr_* T(\NN_P) \to \pr_* \O_{\NN_P} \tensor T_{G/P} \to 0\] 

\begin{lemma}
Let $S_{\p}:= S^{\bullet}(\u_{\p})$. There is a vector space isomorphism $\psi: S_{\p} \tensor (\u_{\p} \dirsum \n_{\p}) \to V_1$.\label{lem:psiisom}
\end{lemma}

\begin{proof}
Note that the map
\[\widetilde{\psi} = (\iota,0) \dirsum (0,\text{id}): S_{\p} \tensor \u_{\p} \dirsum S_{\p} \tensor \n_{\p} \to S_{\p} \tensor \g \dirsum S_{\p} \tensor \n_{\p}\]
induces a map $\psi: S_{\p} \tensor (\u_{\p} \dirsum \n_{\p}) \to V_1$.
 To show that $\psi$ is surjective, let $p \tensor x + q \tensor y \in V_1$ where $p,q \in S_{\p}$ and $x \in \g, y \in \n_{\p}$. Let us write $x = u + w$ where $u \in \u_{\p}$ and $w \in \p$. Using the relations in $V_1$ we get $w = \sum_i p_i \tensor w_i$ where $w_i \in \n_{\p}$, giving $p \tensor w = \sum pp_i \tensor w_i$. Hence $p \tensor x + q \tensor y = p \tensor u + \sum_i pp_i \tensor w_i + q \tensor y$ is in the image of $\psi$. To prove injectivity of $\psi$, suppose that $\psi (p \tensor u + q \tensor y) = 0$. This means that $\tilde{\psi}(p \tensor u + q \tensor y) \in S_\p \tensor \Delta(\p)$, hence there are polynomials $r_i \in S_{\p}$ and elements $x_i \in \p$ (we can assume that $x_i$ are linearly independent) such that $p \tensor u + q \tensor y + \sum r_i \tensor (x_i + \ad(x_i)) = 0$. Now it is easy to see that we should have $r_i = 0$, and it follows that $p = q = 0$ since $u \in \u_{\p}$ and $y \in \n_{\p}$.
\end{proof}

\begin{corollary}
For any $j \geq 0$, the map $\wedge^j\psi: \wedge^j_{S_{\p}} (S_{\p} \tensor (\u_{\p} \dirsum \n_{\p})) \to \wedge^j_{S_{\p}} V_1$ is an isomorphism of vector spaces.
\end{corollary}

This isomorphism preserves the $k$-grading. Hence we obtain:

\begin{corollary}\label{cor:newbasis}
There is a $\h$-equivariant, $k$-graded vector space isomorphism \[ V_{j,k} \cong \bigoplus_{r_2 + r_3 = j, r_1 - r_3 = k/2} S^{r_1}(\u_{\p}) \tensor \bigwedge^{r_2}\u_{\p} \tensor \bigwedge^{r_3}\n_{\p} \]
\end{corollary}

The isomorphism is also compatible with the $\p$-structure in the following sense:

\begin{lemma}
There is a $\p$-module structure on $\wedge_{S_{\p}}^j(S_{\p} \tensor (\u_{\p} \dirsum \n_{\p}))$ such that $\wedge^j \psi$ is $\p$-equivariant.
\end{lemma}

\begin{proof}
It is sufficient to prove the case $j = 1$. For $u \in \u_{\p}$ and $x \in \p$, write $[x,u] = u_1 + x_1$ where $u_1 \in \u_{\p}$ and $x_1 \in \p$. Then, we define a $\p$-module structure on $S_{\p} \tensor (\u_{\p} \dirsum \n_{\p})$ by 
\[{x \cdot (u, x') = (u_1, - \ad(x_1) + [x,x'])}.\]
This induces an obvious $\p$-module structure on $S_{\p} \tensor (\u_{\p} \dirsum \n_{\p})$ and on $\wedge^j_{S_{\p}} (S_{\p} \tensor (\u_{\p} \dirsum \n_{\p}))$. By construction $\psi$ (resp. $\wedge^j\psi$) is $\p$-equivariant. 
\end{proof}

\paragraph{Example} We set $\g = \sl_3$ with usual Chevalley generators $e_i,h_i,f_i$. Pick the vector $e_2^2 \tensor e_1 \tensor f_1 \wedge f_2 \in S^2(\u) \tensor \u \tensor \wedge^2 \n$. We have $f_1 \cdot e_2 = 0$, $f_1 \cdot f_1 = 0$ and $f_1 \cdot f_2 = f_{12}$. We also have $[f_1, e_1] = -h_1$ and $\ad_{\n}(h_1) = -2e_1 \tensor f_1 + e_2 \tensor f_2 - e_{12} \tensor f_{12}$. Hence we get 
\[f_1 \cdot( e_2^2 \tensor e_1 \tensor f_1 \wedge f_2) = e_2^2e_{12} \tensor f_1 \wedge f_2 \wedge f_{12} + e_2^2 \tensor e_1 \tensor f_1 \wedge f_{12}.\]

\subsection{An improvement of the sheaf-cohomology BGG algorithm}\label{sec:Vjk-nice}

The presentation of $V_{j,k}$ of corollary~\ref{cor:newbasis} is very helpful for computing the cohomology of the BGG resolution. However, the resulting $\mathfrak p$-module structure on $V_{j,k}$is too complicated to work with directly. Instead, we will write $V_{jk}$ as cokernel of the map $\Delta\colon T_{jk}\to M_{jk}$, where 
\begin{align*}
    M_{j,k} &= \bigoplus_{r=0}^j S^{j-r+k/2}(\u_\p)\tensor \bigwedge^r\g\tensor \bigwedge^{j-r}\n_\p\\
    T_{j,k} &= \bigoplus_{r=0}^{j-1} S^{j-r+k/2}(\u_\p)\tensor \bigwedge^r\g\tensor \bigwedge^{j-r-1}\n_\p
\end{align*}
The $\mathfrak p$-module structures on $M_{j,k}$ and $T_{j,k}$ are the obvious ones. 
We can then describe the cohomology of the BGG complex $\BGG^{\bullet}(V_{jk},\lambda)$ in terms of $\BGG^{\bullet}(M_{jk},\lambda)$ and the map $\Delta$. This is computationally useful because the $\p$-module structure on $M_{jk}$ is much simpler, despite the fact that the dimension of $M_{jk}$ is larger. 

The map $\Delta$ induces a short exact sequence of $\p$-modules
\[
    \begin{tikzcd}
    0 \arrow[r] & T_{j,k}\arrow[r, "\Delta"] & M_{j,k} \arrow[r,"\varpi"] & V_{j,k} \arrow[r] & 0
    \end{tikzcd}
\]

Note that the $U(\n)$ action on $V_{jk}$ is defined in terms of that on $M_{jk}$. Suppose $\mathcal F\in U(\n)$, and $X\in M_{jk}$. Then we have that $\mathcal F\cdot\varpi(X) = \varpi(\mathcal F\cdot X)$. Now let us choose a linear section $\sigma$  of $\varpi$ (respecting the $\h$-grading). Then if $Y\in V_{jk}$ we have that $\mathcal F\cdot Y= \varpi(\mathcal F\cdot\sigma(Y))$. Since the differential in the BGG complex is given by multiplication by elements of $U(\n)$ on weight components of $V_{jk}$ and $M_{jk}$ this implies that we have equality $d_V = \varpi d_M\sigma$ with $d_V$ and $d_M$ the respective differentials on $\BGG^{\bullet}(V_{j,k},\lambda)$ and $\BGG^{\bullet}(M_{j,k},\lambda)$. 

We will describe how one can find a section $\varpi$ and the associated $\sigma$. We will always use a basis of $M_{jk}$ and $T_{jk}$ induced from a Chevalley basis on $\mathfrak g$. In practice we realize $V_{jk}$ as $\operatorname{coker} \Delta = \ker \Delta^\top$, where the transposition is with respect to the Chevalley basis. For practical reasons we choose a basis of $\operatorname{coker}\Delta$ with integer coefficients to describe the map $\varpi$. Furthermore instead of using a true differential we take, 
\begin{equation}
    \widetilde{d}_V:= \varpi d_M\varpi^\top.
\end{equation}
Since $\varpi$ is a priori just a basis, we have that $\varpi\varpi^\top\neq\Id$, hence ${\widetilde{d}_V}^2\neq 0$, and ${\widetilde{d}_V}$ is not a differential. We could choose $\varpi$ such that $\varpi\varpi^\top =\Id$ but this is in general impossible with integer entries, and by using integer entries we get a big computational benefit for exact linear algebraic operations. Nevertheless we do have that
\[
    \dim \ker \widetilde{d}_V-\dim \im \widetilde{d}_V = \dim \ker {d}_V-\dim \im {d}_V
\]
so that $\widetilde{d}_V$ can be used to compute the cohomology of $\BGG^{\bullet}(V_{jk},\lambda)$. This is because $\varpi^\top$ can be transformed to a section by an isomorphism: $\sigma:=\varpi^\top(\varpi\varpi^\top)^{-1}$. Hence the images of $\widetilde{d}_V$ and $d_V$ are identical while the kernels have the same dimension. 

The cokernel of $\Delta$ can be efficiently computed using the isomorphism $\psi$ of lemma~\ref{lem:psiisom}. The fact that $\psi$ is an isomorphism gives us a convenient subspace $\im \psi\subset M_{jk}$ such that $\varpi|_{\im\psi}$ is an isomorphism, but this is not yet a full description of $\ker\Delta^\top$. The subspace $\im\psi$ has a basis $\{x_i\}_{i\in I}$, and we can define $\varpi(x_i) = e_i$ with $e_i$ an elementary basis vector. Each element $i\in I$ then corresponds to a row of $\Delta^\top$. To obtain a full description of $\varpi$ we simply need to write the remaining rows of $\Delta^\top$ as a linear combination of the columns corresponding to the indices in $I$, i.e. if $j\notin I$ then we write
\begin{equation}\label{eq:defvarpi}
    \Delta^\top_j = \sum_{i\in I}c_i \Delta_i,\qquad \varpi(x_j):=\sum_{i\in I}c_i\varpi(x_i)
\end{equation}

Thus to compute $\varpi$ we need to solve a linear system, which is much more efficient than directly computing the kernel of $\Delta^\top$. In Summary:

\begin{proposition}
To compute the BGG cohomology of $V_{jk}$ we use
\[
    \dim \H^i(\BGG^{\bullet}(V_{jk})) = \dim \ker \varpi d_{M^i_{jk}}\varpi^\top -\dim \im \varpi d^{i-1}_{M_{jk}}\varpi^\top 
\]
with $\varpi$ as in~\eqref{eq:defvarpi}.
\end{proposition}

\begin{remark} Since $\Delta$ is also an equivariant map, we can also find a retract $\rho$ such that $d_T=\rho d_M\Delta$. This also tells us that $d_M = \Delta d_T\rho+\sigma d_V\varpi$ such that the induced short exact sequence on the BGG complexes splits. Hence in particular,
\[
    \H^i(\BGG^{\bullet}(M_{jk}))\cong \H^i(\BGG^{\bullet}(T_{jk}))\dirsum \H^i(\BGG^{\bullet}(V_{jk}))
\]
\end{remark}

\section{Equivariant Hochschild cohomology of the non-trivial singular block of \texorpdfstring{$\u_q(\sl_3)$}{uq(sl3}}\label{sec:P2}

In this section we will compute the equivariant Hochschild cohomology of the non-trivial singular block of $\u_1(\sl_3)$, corresponding geometrically to equivariant Hochschild cohomology of  $T^*\P^2$. A good understanding of these Hochschild cohomology groups is highly desirable, since it could lead to an explicit description of the geometric cup-product and Gerstenhaber bracket. The Hochschild cohomology $\HH^{\bullet}_{\C^*}(\P^1)$ is already fully understood by \cite{LQ3} (we will recall their results later), and $\P^2$ is the next non-trivial case. 

Let us fix the notation for this section: we set $\g = \sl_3$, $\p = \b \dirsum \C \{ e_1 \}$, and $P$ parabolic subgroup such that $\SL_3/P \cong \P^2$. We obtain $\n_{\p} = \C f_2 \dirsum \C f_{12}$ and $\u_{\p} \cong \g/\p \cong \C e_2 \dirsum \C e_{12}$. We use the same convention as \cite{LQ}, that is $e_1,e_2,f_1,f_2,h_1,h_2$ are the usual Chevalley generators, and we define $f_{12}:= [f_1,f_2]$ and $e_{12}:= [e_2,e_1]$. Also for a root weight $\lambda = n_1 \alpha_1 + n_2 \alpha_2$, let $L_{n_1,n_2}$ be the irreducible $\sl_3$-representation with highest weight $\lambda$. 

We now state an elementary but useful lemma about root weights of $\sl_3$: 

\begin{lemma}\label{lem:useful}
Let $ \lambda = n_1 \alpha_1 + n_2 \alpha_2 $ be a root weight, and assume $n_2 \geq n_1 \geq 0$. Then exactly one of the three cases occur: 

\begin{itemize}
    \item $\lambda$ is dominant.
    \item $s_1 \cdot \lambda = (n_2-n_1-1)\alpha_1 + n_2\alpha_2$ is dominant. 
    \item $\lambda$ is dot-singular. (In this case, $2n_1 = n_2 -1$.)
\end{itemize}
In particular, the cohomology of the corresponding line bundles is concentrated in degree $0$ and $1$.

\end{lemma}

\begin{proof}
If $\lambda$ is not dominant then $n_2 \geq 2n_1 +1$. We obtain $s_1 \cdot \lambda = (n_2-n_1-1)\alpha_1 + n_2\alpha_2$. This is dominant if $2(n_2 - n_1 - 1) \geq n_2$ i.e $n_2 \geq 2n_1 + 2$. Hence, if $\lambda$ and $s_1 \lambda$ are not dominant, we should have $n_2 = 2n_1 + 1$ which implies that $s_1 \cdot \lambda = \lambda$.
\end{proof}

\subsection{Description of \texorpdfstring{$\HH_{\C^*}^{\bullet}(T^*\P^2)$}{HH(TP2)} as \texorpdfstring{$\g$}{g}-module}

The purpose of this subsection is to prove the following theorem:

\begin{theorem}\label{thm1}

Let $\u_{\lambda_{\p}}(\sl_3)$ be the block of the small quantum group corresponding to $\p$. For $s \geq 2$, $\HH^s(\u_{\lambda_{\p}}(\sl_3))$ is given by the following tables:

For $s = 2m+1$ odd: {\setlength{\mathindent}{10pt}
\[
\begin{array}{r|l l}
	{\scriptstyle i+j=1}& L_{m,m} \dirsum L_{m+1,m} \dirsum L_{m,m+1} \dirsum L_{m+1,m+1} &\\
	{\scriptstyle i+j=3} & 0 & L_{m,m} \dirsum L_{m+1,m} \dirsum L_{m,m+1} \dirsum L_{m+1,m+1} \\
	\hline h^{i,j}&{\scriptstyle j-i=1}&{\scriptstyle j-i=3}
\end{array}
\] 

For $s = 2m$ even: 
\[
\begin{array}{r|l l l}
	{\scriptstyle i+j=0}&L_{m,m}&&\\
	{\scriptstyle i+j=2}&0&L_{m,m-1} \dirsum L_{m-1,m} \dirsum L_{m,m}^{2} \dirsum L_{m,m+1} \dirsum L_{m+1,m}&\\
	{\scriptstyle i+j=4}&0&0& L_{m,m}\\
	\hline h^{i,j}&{\scriptstyle j-i=0}&{\scriptstyle j-i=2}&{\scriptstyle j-i=4}
\end{array}
\]

%For $s=0$: \[     \begin{array}{r|l l l}	{\scriptstyle i+j=0}&L_{0,0}&&\\	{\scriptstyle i+j=2}&L_{0,0}&L_{0,0} &\\	{\scriptstyle i+j=4}&L_{0,0}&L_{0,0}& L_{0,0}\\	\hline h^{i,j}&{\scriptstyle j-i=0}&{\scriptstyle j-i=2}&{\scriptstyle j-i=4} \end{array} \]

For $s=1$:
\[
\begin{array}{r|l l}
	{\scriptstyle i+j=1}& L_{0,0} \dirsum L_{1,1} &\\
	{\scriptstyle i+j=3} & L_{0,0} & L_{0,0} \dirsum L_{1,1} \\
	\hline h^{i,j}&{\scriptstyle j-i=1}&{\scriptstyle j-i=3}
\end{array}
\]}
\end{theorem}

For $s=0$, this was computed more generally for all projective space in \cite{LQ2}. For $s=1$ this is an easy computation, hence we assume now that $s \geq 2$ and divide the proof into a series of lemmas. Notice that using the $\sl_2$-action, we only need to compute half of the table, as explained in section $2$. 

\begin{lemma}\label{lemma1}
We have $\H^0(\P^2, \Sym^m(T_{\P^2})) \cong L_{m,m}$.
\end{lemma}
\begin{proof}
We need to compute the BGG complex associated to $\Sym^m(\u_{\p})$. A basis of $\Sym^m(\u_{\p})$ is given by $\{e_2^{m-i}e_{12}^i\}_{0\leq i\leq m}$. If $0 \leq a \leq m$, $e_1^{m-a}e_{12}^{a}$ and $e_1^{a+1}e_{12}^{m-a-1}$ have weights in the same dot-orbit. By lemma~\ref{lem:useful}, one of the weights $\mu$ is dominant and the other weight $\mu'$ is such that $s_1 \cdot \mu'$ dominant (if they coincide, then by definition $\mu$ is dot-singular). It is easy to see that the corresponding map in the BGG complex is non-zero, hence the contributions cancel for $a = 0, \dots, m-1$. For $a=m$ the BGG complex is $\C \to 0$, giving the result. 
\end{proof}

\begin{corollary}[$i=0$, $j=0$]\label{corollaryfunctions}
There is an isomorphism of $G$-modules $\H^0(\NN_P, \mathcal O_{\NN_P}) \cong \bigoplus_{m \geq 0} L_{m,m}$. If $k = 2m$, the $k$-th graded part corresponds to $L_{m,m}$.
\end{corollary}

\begin{proof}
For a vector bundle $\pr: E \to X$ we have $\pr_* \mathcal O_E \cong \Sym^{\bullet}(E^{\vee})$. Hence, \[{\H^0(E, \mathcal O_E) \cong \H^0(X,\Sym^{\bullet}(E^{\vee}))},\] and thus
\[\H^0(\NN_P, \mathcal O_{{\NN}_P}) = \H^0(\P^2, \pr_* \mathcal O_{{\NN}_P}) = \H^0(\P^2,\Sym^{\bullet}(T \P^2)).\] The result then follows from the previous lemma.
\end{proof}

\begin{proposition}[$i=0$, $j=1$]
Let $2m = k$. If $m \geq 2$, we have \[\H^0(\NN_P, T_{\NN_P})^{k} \cong L_{m,m} \dirsum L_{m,m+1} \dirsum L_{m+1,m} \dirsum L_{m+1,m+1}\] 
\end{proposition}

\begin{proof}
First, let us write the weight spaces of the corresponding module $S^m(\u_{\p}) \tensor \u_{\p} \dirsum S^{m+1}(\u_{\p}) \tensor \n_{\p}$ in a table: 
\[\begin{array}{l|l|l|l}
 \text{Weight} & \text{Modules} & \text{Vector} & \text{Value of } a \\ \hline a \alpha_1 + m \alpha_2 & S^{m+1}(\u_{\p}) \tensor \n_{\p} & e_2^{m+1-a}e_{12}^a \tensor f_2 & 0 \leq a \leq m+1 \\
 (a-1) \alpha_1 + m\alpha_2 & S^{m+1}(\u_{\p}) \tensor \n_{\p} & e_2^{m+1-a}e_{12}^a \tensor f_{12} & 0 \leq a \leq m+1 \\ 
 a \alpha_1 + (m+1)\alpha_2 & S^m(\u_{\p}) \tensor \u_{\p} & e_2^{m-a}e_{12}^a \tensor e_2 & 0 \leq a \leq m  \\
 (a+1)\alpha_1 + (m+1)\alpha_2 & S^m(\u_{\p}) \tensor \u_{\p} & e_2^{m-a} e_{12}^a \tensor e_{12} & 0 \leq a \leq m
\end{array}\]
Hence, the set of admissible dominant weights are of the form $n_1\alpha_1 + m \alpha_2$ and $n_1 \alpha_1 + (m+1)\alpha_2$ (for appropriate values of $n_1$). First, we notice that we have $s_2 \cdot (a \alpha_1 + m \alpha_2) = a \alpha_1 + (a - m - 1) \alpha_2$. It is obvious that there are no vectors with this weight (and similarly for the weight $s_2 \cdot (a \alpha_1 + (m+1)\alpha_2)$). Hence the differential $E[\lambda] \to E[s_2 \cdot \lambda] = 0$ is zero. Therefore the BGG complex in our case is simply given by $f_1^t\colon E[\lambda] {\to} E[s_1 \cdot \lambda]$ for some integer $t$.

We begin with the case $\lambda = n_1 \alpha_1 + m \alpha_2$, where $2n_1 \geq m$ meaning $\lambda$ is dominant. We assume $ n_1 \leq m-1$ first. Then, the BGG complex is $d\colon \C^2{\to} \C^2$, where the differential is given by multiplication by $f_1^{2n_1-m-1}$. Let us write the formula on a basis, we have
\begin{align*}
     d(e_2^{m+1-n_1}e_{12}^{n_1} \tensor f_2) =  
     \frac{n_1!}{(m-n_1-1)!}\left(e_2^{n_1+2}e_{12}^{m-n_1-1} \tensor f_2 + (2n_1-m+1) e_2^{n_1+1}e_{12}^{m-n_1} \tensor f_{12}\right)
 \end{align*}
  and 
  \[ d(e_2^{m-n_1}e_{12}^{n_1+1} \tensor f_{12}) = \frac{(n_1+1)!}{(m-n_1-1)!}e_2^{n_1+1}e_{12}^{m-n_1} \tensor f_{12} 
  \]
 Indeed, $f_1 \cdot f_2 = f_{12}, f_1 \cdot f_{12} = f_1 \cdot e_2 = 0$, and the action of $f_1$ is given by 
 \[ f_1 \cdot e_2^{r_1}e_{12}^{r_2} \tensor f_2 = r_2 e_2^{r_1+1}e_{12}^{r_2-1} \tensor f_2 + e_2^{r_1}e_{12}^{r_2} \tensor f_{12} \]
 and by induction the formula for the differential follows. It also follows easily that $d$ is an isomorphism, thus these weights do not contribute to the cohomology. We look at the remaining case: $n_1 \in \{ m,m+1 \}$. For $n_1 = m$ the complex is given by $\C^2 \to \C$, where the basis of $\C^2$ is given by $e_2e_{12}^m \tensor f_2$ and $e_{12}^{m+1}f_{12}$, and the basis of $\C$ by $e_2^{m+1} \tensor f_{12}$. Since $f_1^{m+1} \cdot e_{12}^{m+1} \tensor f_{12} = e_2^{m+1} \tensor f_{12}$, it follows that the differential is non-zero. Hence $L_{m,m}$ appears in $\H^0(\NN_P, T_{\NN_P})^k$ with multiplicity $1$.  For the weight $(m+1)\alpha_1 + m \alpha_2$, the complex is $\C \to 0$, therefore $L_{m+1,m}$ appears in $\H^0(\NN_P, T_{\NN_P})^k$ with multiplicity one. 

Now we look at the case  $\lambda = n_1 \alpha_1 + (m+1)\alpha_2$. We take $\lambda$ dominant and first look at $n_1 \leq m-1$. Let us recall that by definition, the action of $f_1$ on $e_{12}, e_2 \in \wedge^1 \u_{\p}$ is given by $f_1 \cdot e_{12} = e_2$ and $f_1 \cdot e_2 = 0$. The differentials are given by:
\begin{align*}
     d(e_2^{m+1-n_1}e_{12}^{n_1} \tensor e_{12}) = \frac{n_1!}{(m-n_1-1)!}\left(e_2^{n_1+2}e_{12}^{m-n_1-1} \tensor e_{12} + (2n_1-m+1)e_2^{n_1+1}e_{12}^{m-n_1} \tensor e_2\right)
\end{align*}
 and 
\[ d(e_2^{m-n_1}e_{12}^{n_1+1} \tensor e_2) = \frac{(n_1+1)!}{(m-n_1-1)!}e_2^{n_1+1}e_{12}^{m-n_1} \tensor e_2 \]

It follows again that all these weights do not contribute to the cohomology. We are left with the weights $m\alpha_1 + (m+1)\alpha_2$ and $(m+1)\alpha_1 + (m+1)\alpha_2$. In the first case, the complex is $\C^2 \to \C$, where a basis of $\C^2$ is given by $e_{12}^m \tensor e_2$ and $e_2e_{12}^{m-1} \tensor e_{12}$, and where a basis of $\C$ is given by $e_2^m \tensor e_2$. The differential is again surjective, thus $L_{m,m+1}$ appears with multiplicity $1$ in $\H^0(\NN_P, T_{\NN_P})^k$. Finally, the BGG complex for $(m+1)\alpha_1 + (m+1)\alpha_2$ is given by $\C \to 0$, and the proposition follows. 
\end{proof}

\begin{lemma}[$i=1$, $j=1$]
We have $\H^1(\NN_P, T_{\NN_P}) = 0$.
\end{lemma}

\begin{proof}
It follows from the previous proof that $E[\lambda] \to E[s_1 \cdot \lambda]$ is surjective for each $\lambda$ that possibly contribute to the cohomology. 
\end{proof}

\begin{proposition}[$i=0$, $j=2$]
Write $s = 2m$. We have \[ \H^0(\NN_P, \wedge^2T_{\NN_P})^k = L_{m,m-1} \dirsum L_{m-1,m} \dirsum L_{m,m}^{2} \dirsum L_{m,m+1} \dirsum L_{m+1,m}\] 
\end{proposition}

\begin{proof}
We have $i+j+k = s$ hence $k = 2(m-1)$. 
Again, let us first split the module into different weight spaces: 

\[\begin{array}{l|l|l|l}
 \text{Weight} & \text{Modules} & \text{Vector} & \text{Value of } a \\ \hline (a+1) \alpha_1 + (m+1) \alpha_2 & S^{m-1}(\u_{\p}) \tensor \wedge^2 \u_{\p} & e_2^{m-a-1}e_{12}^a \tensor e_2 \wedge e_{12} & 0 \leq a \leq m-1 \\
 (a+1) \alpha_1 + m\alpha_2 & S^m(\u_{\p}) \tensor \u_{\p} \tensor \n_{\p} & e_2^{m-a}e_{12}^a \tensor e_{12} \tensor f_2 & 0 \leq a \leq m \\ 
 a \alpha_1 + m\alpha_2 & S^m(\u_{\p}) \tensor \u_{\p} \tensor \n_{\p} & e_2^{m-a}e_{12}^a \tensor e_2 \tensor f_2 & 0 \leq a \leq m  \\
 a \alpha_1 + m\alpha_2 & S^m(\u_{\p}) \tensor \u_{\p} \tensor \n_{\p} & e_2^{m-a}e_{12}^a \tensor e_{12} \tensor f_{12} & 0 \leq a \leq m  \\
 (a-1) \alpha_1 + m\alpha_2 & S^m(\u_{\p}) \tensor \u_{\p} \tensor \n_{\p} & e_2^{m-a}e_{12}^a \tensor e_2 \tensor f_{12} & 0 \leq a \leq m  \\
  (a-1)\alpha_1 + (m-1)\alpha_2 & S^{m+1}(\u_{\p}) \tensor \wedge^2\n_{\p} & e_2^{m+1-a} e_{12}^a \tensor f_2 \wedge f_{12} & 0 \leq a \leq m +1 
\end{array}\]

As in the previous case, for all the weights $\lambda$ appearing there is no element of weight $s_2 \cdot \lambda $. First, we compute cohomology with dominant weight of the form $\lambda = n_1 \alpha_1 + (m+1)\alpha_2$. We have $s_1 \cdot (n_1 \alpha_1 + (m+1)\alpha_2) = (m-n_1)\alpha_1 + (m+1)\alpha_2$. If we assume $n_1 \neq m, n_1 \geq 2(m+1)$ the BGG complex is $\C \to \C, e_2^{m-n_1} e_{12}^{n_1} \tensor e_2 \wedge e_{12} \mapsto e_2^{n_1} e_{12}^{m-n_1} \tensor e_2 \wedge e_{12}$ and hence no cohomology comes from these weights. Finally, for $n_1 = m$ we get the BGG complex $\C \to 0$, therefore $L_{m,m+1}$ contributes with multiplicity one. 

\indent We now look at dominant weights on the form $ \lambda = n_1 \alpha_1 + (m-1)\alpha_1$. We know that $-1 \leq n_1 \leq m$. Moreover, $s_1 \cdot \lambda = (m-n_1 - 2) \alpha_1 + (m-1)\alpha_2$, and for these weights the BGG complex is $1\colon \C {\to} \C$. Hence, for $-1 \leq n_1 \leq m-1$ the weights do not contribute to the cohomology. The only remaining weight is $\lambda = m\alpha_1 + (m-1)\alpha_2$, the BGG complex is $\C \to 0$. Hence $L_{m,m-1}$ appears with multiplicity one. 

\indent We now look at the case where $\lambda = n_1 \alpha_1 + m\alpha_2$ is dominant. If $n_1 \leq m-2$ we claim that $d$ is invertible, hence these weights do not contribute to the cohomology. Indeed, the matrix for $d$ is given by 

\[ \begin{pmatrix}
S &  0 & 0 & 0\\
-T & S & 0 & 0\\
T & 0 & S & 0\\
-\frac{(2n_1-m)T}{m-n_1+2} & -T & -T & S
\end{pmatrix}, \]

where $T = (2n_1-m+1)(n_1-1) \dots (m-n_1+2)$ and $S = (n_1-1)\dots(m-n_1+1)$. These coefficients are obtained by induction, for example the coefficients in the first column can be deduced from 
\begin{align*}
&f_1 \cdot ( e_2^{m-n_1}e_{12}^{n_1-1} \tensor e_{12} \tensor f_2)\\
 &\qquad = (n_1-1) e_2^{m-n_1+2}e_{12}^{n_1-2} \tensor e_{12} \tensor f_2 - e_2^{m-n_1}e_{12}^{n_1-1} \tensor e_2 \tensor f_2 + e_2^{m-n_1}e_{12}^{n_1-1} \tensor e_{12} \tensor f_{12}&
\end{align*}

Since $S \neq 0$ it follows that $d$ is an isomorphism.\\
\indent The remaining cases are $n_1 \in \{m-1,m,m+1\}$. For the case of $\lambda = (m+1)\alpha_1 + m\alpha_2$ the BGG complex is $\C \to 0$. \\ 
\indent If $\lambda = (m-1)\alpha_1 + m \alpha_2$, $s_1 \cdot \lambda = m \alpha_2$ and there is no vector of this weight of the form $e_2^{m-a}e_{12}^a\tensor e_{12} \tensor f_2$. Hence the BGG complex is $d\colon \C^4 {\to} \C^3$ and we see easily using the previous formula that $d$ is surjective. \\
\indent Finally, if $\lambda = m\alpha_1 + m \alpha_2$, the BGG complex is $\C^3 \to \C$, with basis of $\C^3$
\[
e_2e_{12}^{m-1} \tensor e_{12} \tensor f_2,\quad e_{12}^m \tensor e_2 \tensor f_2,\quad  e_{12}^m \tensor e_{12} \tensor f_{12},
\]
and for $\C^1$ basis $e_2 \tensor e_2 \tensor f_{12}$. It follows again from the explicit formula that $d$ is surjective, hence $L_{m,m}$ appears with multiplicity $2$. 
\end{proof}

\begin{lemma}[$i=1$, $j=2$]
We have $\H^1(\NN_P, \wedge^2 T \NN_P) = 0$.
\end{lemma}

\begin{proof}
It follows from the previous proof that $\H^1 = 0$, since for each dominant weight $d$ was always surjective. 
\end{proof}

\begin{lemma}[$i=2$, $j=2$]

We have $\H^2(\NN_P, \wedge^2 T \NN_P) = 0$. 
\end{lemma}

\begin{proof}
The module $V_{2,m}$ is given by \[ (\Sym^m(\u_{\p}) \tensor \wedge^2\u_{\p}) \dirsum (\Sym^{m+1}(\u_{\p}) \tensor \u_{\p} \tensor \n_{\p}) \dirsum (\Sym^{m+1}(\u_{\p}) \tensor \wedge^2 \n_{\p}). \]By lemma \ref{lem:useful}, we only need to check weight $\lambda = n_1 \alpha_1 + n_2 \alpha_2$ with $n_1 < 0$ or $n_2 < 0$. There is a single weight $\mu$ with such coefficients, given by $\mu = -\alpha_1 + (m+1)\alpha_2$. Since $s_1 \cdot \mu = (m+1)\rho$ is dominant, $\H^2 = 0$ as claimed. (If $m=0$, there is also the exceptional case $e_2 \tensor f_2 \wedge f_{12}$ with weight $-\alpha_1 - \alpha_2 = - \rho$. But such a weight is dot-singular.)
\end{proof}

\subsection{Relation to the nilpotent cone}

In this subsection, we describe $\HH_{\C^*}(T^*\P^2)$ as a module over $\C[\N]$, where $\N \subset \sl_3$ is the nilpotent cone.

%\begin{theorem}\label{thm:springerresolution} There is a resolution of singularities $f: \NN = T^*(G/B) \to \N$, with $f_* \mathcal O_{\NN} = \C[\N]$. This resolution is called the \emph{Springer resolution}. \end{theorem}

%\begin{theorem} Let $\lambda$ be a partition of $n$, and $P$ a standard parabolic subgroup corresponding to $\lambda$. Then, there is a resolution of singularities $f: \NN_P = T^*(G/P) \to \overline{\O_{\lambda}}$, where $\O_{\lambda} =: \O_{\p}$ is the nilpotent orbit associated to $\lambda$. We also have $f_* \mathcal O_{\NN_P} = \C[\O_{\p}]$. \end{theorem}

%We have an explicit description of $\C[\N]$ as an induced $B$-module: \begin{proposition}[\cite{AJ}]\label{prop:nilpotentinduced} We have $\C[\N] = \Ind_B^G \Sym^{\bullet}(\u)$. \end{proposition} \begin{proof} We have the equalities \[ \C[\N] = H^0(\NN, \O_{\NN}) = \H^0(G/B, \Sym^{\bullet} (T(G/B))) = \Ind_B^G\Sym^{\bullet}(\u)  \] where the first equality follows from the theorem \ref{thm:springerresolution}, the second equality follows from the proof of corollary \ref{corollaryfunctions} and the last equality follows from the definition of the induction functor. \end{proof}

% The same proof gives:\begin{proposition}There is an isomorphism of $G$-algebras $\C[\mathcal O_{\p}] = \Ind_P^G \Sym^{\bullet}(\u_{\p})$. \end{proposition}

The following theorem by Ginzburg and Kumar relates $\u_q(\g)$ with $\N$: 
\begin{theorem}[\cite{GK}]
There is an isomorphism of algebras  $\H^{2\bullet}(\u_q(\g), \C) \cong \C[\N]$. Moreover we have that ${\H^{2\bullet+1}(\u_q(\g),\C) = 0}$. 
\end{theorem}   

Since $\H^{\bullet}(\u_q(\g), \C)$ acts on $\HH^{\bullet}(\u_q(\g))$, for each block the $G$-representation $\HH^{\bullet}(\u_{\lambda_{\p}}(\g))$ is actually a $G \times \C^*$-equivariant coherent sheaf on the nilpotent cone. There is a $\H^0(G/P, \pr_* \O_{\NN_P})$-module structure on $\HH^{\bullet}(\u_{\lambda_{\p}}(\g))$, hence a $\C[\N]$-module structure. In \cite{LQ3}, Lachowska-Qi obtained the $\C[\N]$-module of $\HH^{\bullet}(\u_0(\sl_2))$: 

\begin{theorem}[\cite{LQ3}]
As module over $\C[\N]$, the even Hochschild cohomology of $\u_0(\sl_2)$ is given by the following table: 

\[
\begin{array}{r|l l}
	{\scriptstyle i+j=0}& \C[\N]  &\\
	{\scriptstyle i+j=2} & \C & \C[\N] \\
	\hline h^{i,j}&{\scriptstyle j-i=0}&{\scriptstyle j-i=2}
\end{array}
\]

The odd Hochschild cohomology is concentrated in bidegree $(0,1)$ and given by 
\[ \HH^{2\bullet+1}(\u_0(\sl_2)) \cong \C[\N]_+[1] \oplus \C[\N][-1],
\] where $[1]$ is a shift in the $s$-grading and $\C[\N]_+$ is the augmentation ideal.
\end{theorem}

We would like to obtain a similar result when $\g = \sl_3$ and $\p = \b \dirsum \C e_1$, hence $G/P \cong \P^2$. Let $\u_1:= \u_{\lambda_{\p}}(\g)$ and $Y:= \mathcal O_{\p}$. The corollary \ref{corollaryfunctions} and the previous discussion implies:

\begin{proposition}
There is a graded $G$-module isomorphism, $\C[Y] \cong \bigoplus_{m \geq 0} L(m,m)$.
\end{proposition} 

We want to understand the $\C[Y]$-module structure on $\HH^{\bullet}(\u)$. Using that $\C[\NN_P] \cong \C[Y]$, we can reduce the problem to understanding the $\O_{\NN_P}$-module structure on the cohomology $\HH_{\C^*}^{\bullet}(\NN_P)$. Using the equivalence $\O_{\NN_P}\text{--mod} \cong \pr_* \O_{\NN_P} \text{--mod}$, we can reduce the problem to understand the maps \[\H^0(\pr_* \O_{\NN_P}) \tensor \left( \bigoplus_{i+j+k=s} \H^i(\P^2, \mathcal V_{jk}) \right) \to \bigoplus_{i+j+k=s+2} \H^i(\P^2, \mathcal V_{jk}),\] where $\mathcal V_{jk}$ is the vector bundle associated to the module $V_{jk}$ introduced earlier. Since $\C[Y]$ is generated by its degree one elements $\C[Y]_1 \cong L_{1,1} \cong \sl_3$, it is enough to compute the maps \[L_{1,1} \tensor \left(\bigoplus_{i+j+k=s} \H^i(\P^2, \mathcal V_{jk})\right) \to  \bigoplus_{i+j+k=s+2} \H^i(\P^2, \mathcal V_{jk}).\] 
We notice that for $m \geq 2$, the only non-zero contribution to $\HH^m(\u_1)$ is when $i=0$. By definition, if $E$ is a $\mathfrak p$-module, and $v \in E^{\n_P}[\mu]$ for some $\mu \in P^+$, then $v$ is the BGG representative corresponding to an irreducible representation $L_{\mu} \subset \H^0(G/P, \mathcal E)$. This vector can be found from the representation as follows: $L_{\mu}$ has a highest weight vector $w$, corresponding to a section $s_w: G/P \to E$. Then, $s_w(eP) = v$ is the corresponding vector. This is by definition, because $\H^0(G/P, E) = \text{Ind}_P^G E$. In particular, it follows that for multiplying global sections we can multiply the corresponding BGG representative together (since the multiplication is an equivariant map, we only need to know what happens at a point, for example at $eP$). 

Let $W_m = L_{m,m} \dirsum L_{m+1,m} \dirsum L_{m,m+1} \dirsum L_{m+1,m+1}$. This is the summand of bidegree $(i,j) = (0,1)$ in $\HH^{2m+1}(\u_1)$. Since the $\sl_2$-action obtained from the Poisson bivector field commutes with the $\C[Y]$-action, we just need to understand the $\C[Y]$-module structure on this summand, and the summand of bidegree $(i,j) = (0,3)$ will be isomorphic to $W_m$. Let $W_m^{1} = L_{m,m}, W_m^{2} = L_{m+1,m}, W_m^3 = L_{m,m+1}$ and $W_m^4 = L_{m+1,m+1}$. 

\begin{proposition}
The maps $L_{1,1} \tensor W_m \to W_{m+1}$ is given by $L_{1,1} \tensor W_m^l \to W_{m+1}^l \subset W_{m+1}$, where the first map is the projection onto the corresponding isotopic component.
\end{proposition}

\begin{proof}
This directly follows from the representatives we computed in the proof of theorem \ref{thm1} and the discussion before. Let us remark that the projection onto the isotopic component is only defined up to multiplication by a non-zero scalar, but it can be fixed by choosing a highest weight vector $v_m$ in each $W_m^l$ such that the projection sends $v_m$ to $v_{m+1}$. 
\end{proof}

A similar proposition holds for $s = 2m$. Now we describe the corresponding $\C[Y]$-modules. We define a module $M_1$ as $(M_1)_m = L_{m+1,m}$ if $m \geq 1$ and $(M_1)_0 = 0$. The $\C[Y]$-module structure is obtained from the maps $L_{1,1} \tensor W_m^2 \to W_{m+1}^2$. Let $M_2$ be the module such that $(M_2)_m = W_m^3$ if $m \geq 1$ and $0$ otherwise, with similar $\C[Y]$-module structure. We also define shifted modules $N_1, N_2$ by $(N_1)_{m} = (M_1)_{m-1}$ and $(N_2)_m = (M_2)_{m-1}$, with the same $\C[Y]$-module structure. Finally, let $\C[Y]_+$ be the ideal sheaf of the origin $0 \in Y$, and $\C_0$ be the skyscraper sheaf at $0$. We obtain the following description of the Hochschild cohomology as $\C[Y]$-modules:

\begin{proposition}
The following tables give the even (resp. odd) Hochschild cohomology groups of $\u_1(\sl_3)$, as $\C[Y]$-modules: 

$\displaystyle
\begin{array}{r|l l l}
	{\scriptstyle i+j=0}&\C[Y]&&\\
	{\scriptstyle i+j=2}& \C_0 & \C[Y] \dirsum N_1 \dirsum N_2 \dirsum \C[Y]_+ \dirsum M_1 \dirsum M_2 &\\
	{\scriptstyle i+j=4}&\C_0&\C_0& \C[Y] \\
	\hline h^{i,j}&{\scriptstyle j-i=0}&{\scriptstyle j-i=2}&{\scriptstyle j-i=4}
\end{array}
$ \\ \\ 

$\displaystyle
\begin{array}{r|l l}
	{\scriptstyle i+j=1}& \C[Y] \dirsum M_1 \dirsum M_2 \dirsum \C[Y]_+ &\\
	{\scriptstyle i+j=3} & \C_0 & \C[Y] \dirsum M_1 \dirsum M_2 \dirsum \C[Y]_+  \\
	\hline h^{i,j}&{\scriptstyle j-i=1}&{\scriptstyle j-i=3}
\end{array}
$ 
\end{proposition}

%It would be interesting to have a geometric description of the modules $M_1$ and $M_2$. 

\section{Computation of \texorpdfstring{$\z_0(\u_q(\g))$}{z0(uq(g))} for type \texorpdfstring{$G_2,B_3,C_3$}{G2,B3,C3} and \texorpdfstring{$A_4$}{A4}}\label{sec:computerComputations}

We now present our results obtained using the algorithm described in \cite{HV}, using the description of $V_{jk}$ introduced in section~\ref{sec:Vjk-nice}.

\subsection{Type \texorpdfstring{$G_2$}{G2}}

For type $G_2$, we let $\alpha_1$ be the short root and $\alpha_2$ be the long root. 

\begin{theorem}

The center of the principal block of the small quantum group for $\g$ of type $G_2$ has dimension $91$. The bigraded dimensions are given by the following table:

\[
\begin{array}{r|l l l l l l l}
	{\scriptstyle i+j=0}&\mathbb{C}&&&&&&\\
	{\scriptstyle i+j=2}&\mathbb{C}^{2}&\mathbb{C}&&&&&\\
	{\scriptstyle i+j=4}&\mathbb{C}^{2}&\mathbb{C}^{2}&\mathbb{C}&&&&\\
	{\scriptstyle i+j=6}&\mathbb{C}^{2}&\mathbb{C}^{2}&\mathbb{C}^{2}&\mathbb{C}&&&\\
	{\scriptstyle i+j=8}&\mathbb{C}^{2}&\mathbb{C}^{2} \dirsum L_{2,1}^{2}&\mathbb{C}^{2}\dirsum L_{2,1}&\mathbb{C}^{2}&\mathbb{C}&&\\
	{\scriptstyle i+j=10}&\mathbb{C}^{2}&\mathbb{C}^{3} \dirsum L_{2,1}&\mathbb{C}^{2} \dirsum L_{2,1}^{2}&\mathbb{C}^{2}&\mathbb{C}^{2}&\mathbb{C}&\\
	{\scriptstyle i+j=12}&\mathbb{C}&\mathbb{C}^{2}&\mathbb{C}^{2}&\mathbb{C}^{2}&\mathbb{C}^{2}&\mathbb{C}^{2}&\mathbb{C}\\
	\hline h^{i,j}&{\scriptstyle j-i=0}&{\scriptstyle j-i=2}&{\scriptstyle j-i=4}&{\scriptstyle j-i=6}&{\scriptstyle j-i=8}&{\scriptstyle j-i=10}&{\scriptstyle j-i=12}
\end{array}
\]

\end{theorem}

The trivial representation has multiplicity $49$ in total. The table confirms the conjecture \ref{conj:diag} for $\g$ of type $G_2$, hence for all complex simple Lie algebras of rank $2$ using previous work from \cite{LQ} and \cite{LQ2}. We also computed the center of the singular blocks of type $\u_1(\g_2)$: 

\[
\begin{array}{r|l l l l l l}
	{\scriptstyle i+j=0}&\mathbb{C}&&&&&\\
	{\scriptstyle i+j=2}&\mathbb{C}&\mathbb{C}&&&&\\
	{\scriptstyle i+j=4}&\mathbb{C}&\mathbb{C}&\mathbb{C}&&&\\
	{\scriptstyle i+j=6}&\mathbb{C}&\mathbb{C} \oplus L_{2,1}&\mathbb{C}&\mathbb{C}&&\\
	{\scriptstyle i+j=8}&\mathbb{C}&\mathbb{C} \dirsum L_{2,1}&\mathbb{C}\dirsum L_{2,1}&\mathbb{C}&\mathbb{C}&\\
	{\scriptstyle i+j=10}&\mathbb{C}&\mathbb{C}^{3} \dirsum L_{2,1}&\mathbb{C} \dirsum L_{2,1}&\mathbb{C}&\mathbb{C}&\\
	\hline h^{i,j}&{\scriptstyle j-i=0}&{\scriptstyle j-i=2}&{\scriptstyle j-i=4}&{\scriptstyle j-i=6}&{\scriptstyle j-i=8}&{\scriptstyle j-i=10}
\end{array}
\]

and similarly for $\u_2(\g_2)$: 

\[
\begin{array}{r|l l l l l l}
	{\scriptstyle i+j=0}&\mathbb{C}&&&&&\\
	{\scriptstyle i+j=2}&\mathbb{C}&\mathbb{C}&&&&\\
	{\scriptstyle i+j=4}&\mathbb{C}&\mathbb{C}&\mathbb{C}&&&\\
	{\scriptstyle i+j=6}&\mathbb{C}&\mathbb{C} \oplus L_{2,1}&\mathbb{C} \oplus L_{2,1} &\mathbb{C}&&\\
	{\scriptstyle i+j=8}&\mathbb{C}&\mathbb{C} &\mathbb{C}\dirsum L_{2,1}&\mathbb{C}&\mathbb{C}&\\
	{\scriptstyle i+j=10}&\mathbb{C}&\mathbb{C}^{3} \dirsum L_{2,1}&\mathbb{C} \dirsum L_{2,1}&\mathbb{C}&\mathbb{C}&\\
	\hline h^{i,j}&{\scriptstyle j-i=0}&{\scriptstyle j-i=2}&{\scriptstyle j-i=4}&{\scriptstyle j-i=6}&{\scriptstyle j-i=8}&{\scriptstyle j-i=10}
\end{array}
\]

\subsection{Type \texorpdfstring{$B_3$}{B3} and \texorpdfstring{$C_3$}{C3}}

\begin{definition}\label{singularblocks} For a subset $J \subset I$, let $\u_J$ be the block corresponding to a singular weight $\lambda$ with stabiliser generated by $\langle s_j: j \in J \rangle $. \end{definition}

\begin{theorem}
In type $B_3$, the center of the principal block is given by the following table:

\resizebox{0.9\textwidth}{!}{\setlength{\mathindent}{0pt}
\[ \begin{array}{r|l l l l l l l l l }
	{\scriptstyle i+j=0}&\mathbb{C}&&&&&&&&\\
	{\scriptstyle i+j=2}&\mathbb{C}^{3}&\mathbb{C}&&&&&&&\\
	{\scriptstyle i+j=4}&\mathbb{C}^{5}&\mathbb{C}^{3}&\mathbb{C}&&&&&&\\
	{\scriptstyle i+j=6}&\mathbb{C}^{7}&\mathbb{C}^{6}&\mathbb{C}^{3}&\mathbb{C}&&&&&\\
	{\scriptstyle i+j=8}&\mathbb{C}^{8}&\mathbb{C}^{10}&\mathbb{C}^{6}&\mathbb{C}^{3}&\mathbb{C}&&&&\\
	{\scriptstyle i+j=10}&\mathbb{C}^{8}&\mathbb{C}^{14}\dirsum L_{1,1,1}&\mathbb{C}^{10}&\mathbb{C}^{6}&\mathbb{C}^{3}&\mathbb{C}&&&\\
	{\scriptstyle i+j=12}&\mathbb{C}^{7}&\mathbb{C}^{15}\dirsum L_{1,1,1}^{3}&\mathbb{C}^{14}\dirsum L_{1,1,1}^{4}&\mathbb{C}^{10}\dirsum L_{1,1,1}&\mathbb{C}^{6}&\mathbb{C}^{3}&\mathbb{C}&&\\
	{\scriptstyle i+j=14}&\mathbb{C}^{5}&\mathbb{C}^{12}\dirsum L_{1,1,1}^{2}&\mathbb{C}^{15}\dirsum L_{1,1,1}^{6}\dirsum L_{1,2,2}&\mathbb{C}^{14}\dirsum L_{1,1,1}^{4}&\mathbb{C}^{10}&\mathbb{C}^{6}&\mathbb{C}^{3}&\mathbb{C}&\\
	{\scriptstyle i+j=16}&\mathbb{C}^{3}&\mathbb{C}^{8}&\mathbb{C}^{12}\dirsum L_{1,1,1}^{2}&\mathbb{C}^{15}\dirsum L_{1,1,1}^{3}&\mathbb{C}^{14}\dirsum L_{1,1,1}&\mathbb{C}^{10}&\mathbb{C}^{6}&\mathbb{C}^{3}&\\
	{\scriptstyle i+j=18}&\mathbb{C}&\mathbb{C}^{3}&\mathbb{C}^{5}&\mathbb{C}^{7}&\mathbb{C}^{8}&\mathbb{C}^{8}&\mathbb{C}^{7}&\mathbb{C}^{5}&\mathbb{C}^{3}\\
	\hline h^{i,j}&{\scriptstyle j-i=0}&{\scriptstyle j-i=2}&{\scriptstyle j-i=4}&{\scriptstyle j-i=6}&{\scriptstyle j-i=8}&{\scriptstyle j-i=10}&{\scriptstyle j-i=12}&{\scriptstyle j-i=14}&{\scriptstyle j-i=16}
\end{array} \]}

The invariant part of the center coincide with the invariant part of the center of the principal block for type $C_3$. 
\end{theorem}

We also computed the corresponding diagonal coinvariants for type $B_3$, and confirmed conjecture \ref{conj:diag} for $B_3$ and $C_3$. We notice that the vector space isomorphism $\z(\u_0(\mathfrak{so}_{7}))^{\mathfrak{so}_{7}} \cong \z(u_0(\mathfrak{sp}_6))^{\mathfrak{sp}_6}$ was also predicted by the conjecture since they have isomorphic Weyl group. We now give the answer for other blocks, and for simplicity we only give the answer when $\g$ is of type $B_3$.

\begin{theorem}
The center of $\u_1(\mathfrak {so}_7)$ and $\u_2(\mathfrak {so}_7)$ are isomorphic as bigraded vector spaces, and given by the following table: 
\setlength{\mathindent}{0pt}
\[ 
\begin{array}{r|l l l l l l l l l}
	{\scriptstyle i+j=0}&\mathbb{C}&&&&&&&&\\
	{\scriptstyle i+j=2}&\mathbb{C}^{2}&\mathbb{C}&&&&&&&\\
	{\scriptstyle i+j=4}&\mathbb{C}^{3}&\mathbb{C}^{2}&\mathbb{C}&&&&&&\\
	{\scriptstyle i+j=6}&\mathbb{C}^{4}&\mathbb{C}^{4}&\mathbb{C}^{2}&\mathbb{C}&&&&&\\
	{\scriptstyle i+j=8}&\mathbb{C}^{4}&\mathbb{C}^{6}&\mathbb{C}^{4}&\mathbb{C}^{2}&\mathbb{C}&&&&\\
	{\scriptstyle i+j=10}&\mathbb{C}^{4}&\mathbb{C}^{7}\dirsum L_{1,1,1}&\mathbb{C}^{6} \dirsum L_{1,1,1}&\mathbb{C}^{4}&\mathbb{C}^{2}&\mathbb{C}&&&\\
	{\scriptstyle i+j=12}&\mathbb{C}^{3}&\mathbb{C}^{6} \dirsum L_{1,1,1}&\mathbb{C}^{7} \dirsum L_{1,1,1}^{3} \dirsum L_{1,2,2}&\mathbb{C}^{6} \dirsum L_{1,1,1}&\mathbb{C}^{4}&\mathbb{C}^{2}&\mathbb{C}&&\\
	{\scriptstyle i+j=14}&\mathbb{C}^{2}&\mathbb{C}^{4}&\mathbb{C}^{6} \dirsum L_{1,1,1}&\mathbb{C}^{7} \dirsum L_{1,1,1}&\mathbb{C}^{6}&\mathbb{C}^{4}&\mathbb{C}^{2}&\mathbb{C}&\\
	{\scriptstyle i+j=16}&\mathbb{C}&\mathbb{C}^{2}&\mathbb{C}^{3}&\mathbb{C}^{4}&\mathbb{C}^{4}&\mathbb{C}^{4}&\mathbb{C}^{3}&\mathbb{C}^{2}&\mathbb{C}\\
	\hline h^{i,j}&{\scriptstyle j-i=0}&{\scriptstyle j-i=2}&{\scriptstyle j-i=4}&{\scriptstyle j-i=6}&{\scriptstyle j-i=8}&{\scriptstyle j-i=10}&{\scriptstyle j-i=12}&{\scriptstyle j-i=14}&{\scriptstyle j-i=16}
\end{array}
\] 

The center of $\u_3(\mathfrak {so}_7)$ is given by the following table: 
\[ 
\begin{array}{r|l l l l l l l l l}
	{\scriptstyle i+j=0}&\mathbb{C}&&&&&&&&\\
	{\scriptstyle i+j=2}&\mathbb{C}^{2}&\mathbb{C}&&&&&&&\\
	{\scriptstyle i+j=4}&\mathbb{C}^{3}&\mathbb{C}^{2}&\mathbb{C}&&&&&&\\
	{\scriptstyle i+j=6}&\mathbb{C}^{4}&\mathbb{C}^{4}&\mathbb{C}^{2}&\mathbb{C}&&&&&\\
	{\scriptstyle i+j=8}&\mathbb{C}^{4}&\mathbb{C}^{6} \dirsum L_{1,1,1}&\mathbb{C}^{4}&\mathbb{C}^{2}&\mathbb{C}&&&&\\
	{\scriptstyle i+j=10}&\mathbb{C}^{4}&\mathbb{C}^{7} \dirsum L_{1,1,1}^{2}&\mathbb{C}^{6} \dirsum L_{1,1,1}^{2}&\mathbb{C}^{4}&\mathbb{C}^{2}&\mathbb{C}&&&\\
	{\scriptstyle i+j=12}&\mathbb{C}^{3}&\mathbb{C}^{6} \dirsum L_{1,1,1}^{2}&\mathbb{C}^{7}\dirsum L_{1,1,1}^{4}&\mathbb{C}^{6} \dirsum L_{1,1,1}^{2}&\mathbb{C}^{4}&\mathbb{C}^{2}&\mathbb{C}&&\\
	{\scriptstyle i+j=14}&\mathbb{C}^{2}&\mathbb{C}^{4}&\mathbb{C}^{6} \dirsum L_{1,1,1}^{2}&\mathbb{C}^{7} \dirsum L_{1,1,1}^{2}&\mathbb{C}^{6} \dirsum L_{1,1,1}&\mathbb{C}^{4}&\mathbb{C}^{2}&\mathbb{C}&\\
	{\scriptstyle i+j=16}&\mathbb{C}&\mathbb{C}^{2}&\mathbb{C}^{3}&\mathbb{C}^{4}&\mathbb{C}^{4}&\mathbb{C}^{4}&\mathbb{C}^{3}&\mathbb{C}^{2}&\mathbb{C}\\
	\hline h^{i,j}&{\scriptstyle j-i=0}&{\scriptstyle j-i=2}&{\scriptstyle j-i=4}&{\scriptstyle j-i=6}&{\scriptstyle j-i=8}&{\scriptstyle j-i=10}&{\scriptstyle j-i=12}&{\scriptstyle j-i=14}&{\scriptstyle j-i=16}
\end{array}
\] 

The center of $\u_{12}(\mathfrak {so}_7)$ is as follows: 

\[ \begin{array}{r|l l l l l l l}
	{\scriptstyle i+j=0}&\mathbb{C}&&&&&&\\
	{\scriptstyle i+j=2}&\mathbb{C}&\mathbb{C}&&&&&\\
	{\scriptstyle i+j=4}&\mathbb{C}&\mathbb{C}&\mathbb{C}&&&&\\
	{\scriptstyle i+j=6}&\mathbb{C}^{2}&\mathbb{C}^{2}&\mathbb{C}&\mathbb{C}&&&\\
	{\scriptstyle i+j=8}&\mathbb{C}&\mathbb{C}^{2}&\mathbb{C}^{2}\dirsum L_{1,1,1}\dirsum L_{1,2,2}&\mathbb{C}&\mathbb{C}&&\\
	{\scriptstyle i+j=10}&\mathbb{C}&\mathbb{C}&\mathbb{C}^{2}&\mathbb{C}^{2}&\mathbb{C}&\mathbb{C}&\\
	{\scriptstyle i+j=12}&\mathbb{C}&\mathbb{C}&\mathbb{C}&\mathbb{C}^{2}&\mathbb{C}&\mathbb{C}&\mathbb{C}\\
	\hline h^{i,j}&{\scriptstyle j-i=0}&{\scriptstyle j-i=2}&{\scriptstyle j-i=4}&{\scriptstyle j-i=6}&{\scriptstyle j-i=8}&{\scriptstyle j-i=10}&{\scriptstyle j-i=12}
\end{array} \] 

The center of $\u_{13}(\mathfrak {so}_7)$ is as follows: 

\[ \begin{array}{r|l l l l l l l l}
	{\scriptstyle i+j=0}&\mathbb{C}&&&&&&&\\
	{\scriptstyle i+j=2}&\mathbb{C}&\mathbb{C}&&&&&&\\
	{\scriptstyle i+j=4}&\mathbb{C}^{2}&\mathbb{C}&\mathbb{C}&&&&&\\
	{\scriptstyle i+j=6}&\mathbb{C}^{2}&\mathbb{C}^{3}&\mathbb{C}&\mathbb{C}&&&&\\
	{\scriptstyle i+j=8}&\mathbb{C}^{2}&\mathbb{C}^{3}\dirsum L_{1,1,1}&\mathbb{C}^{3}&\mathbb{C}&\mathbb{C}&&&\\
	{\scriptstyle i+j=10}&\mathbb{C}^{2}&\mathbb{C}^{3}\dirsum L_{1,1,1}&\mathbb{C}^{3}\dirsum L_{1,1,1}^{2}&\mathbb{C}^{3}&\mathbb{C}&\mathbb{C}&&\\
	{\scriptstyle i+j=12}&\mathbb{C}&\mathbb{C}^{2}&\mathbb{C}^{3}\dirsum L_{1,1,1}&\mathbb{C}^{3}\dirsum L_{1,1,1}&\mathbb{C}^{3}&\mathbb{C}&\mathbb{C}&\\
	{\scriptstyle i+j=14}&\mathbb{C}&\mathbb{C}&\mathbb{C}^{2}&\mathbb{C}^{2}&\mathbb{C}^{2}&\mathbb{C}^{2}&\mathbb{C}&\mathbb{C}\\
	\hline h^{i,j}&{\scriptstyle j-i=0}&{\scriptstyle j-i=2}&{\scriptstyle j-i=4}&{\scriptstyle j-i=6}&{\scriptstyle j-i=8}&{\scriptstyle j-i=10}&{\scriptstyle j-i=12}&{\scriptstyle j-i=14}
\end{array} \]

The center of $\u_{23}(\mathfrak {so}_7)$ is as follows: 

\[ \begin{array}{r|l l l l l l}
	{\scriptstyle i+j=0}&\mathbb{C}&&&&&\\
	{\scriptstyle i+j=2}&\mathbb{C}&\mathbb{C}&&&&\\
	{\scriptstyle i+j=4}&\mathbb{C}&\mathbb{C}&\mathbb{C}&&&\\
	{\scriptstyle i+j=6}&\mathbb{C}&\mathbb{C}\dirsum L_{1,1,1}&\mathbb{C}\dirsum L_{1,1,1}&\mathbb{C}&&\\
	{\scriptstyle i+j=8}&\mathbb{C}&\mathbb{C}&\mathbb{C}\dirsum L_{1,1,1}&\mathbb{C}&\mathbb{C}&\\
	{\scriptstyle i+j=10}&\mathbb{C}&\mathbb{C}&\mathbb{C}&\mathbb{C}&\mathbb{C}&\mathbb{C}\\
	\hline h^{i,j}&{\scriptstyle j-i=0}&{\scriptstyle j-i=2}&{\scriptstyle j-i=4}&{\scriptstyle j-i=6}&{\scriptstyle j-i=8}&{\scriptstyle j-i=10}
\end{array} \]

\end{theorem}

The isomorphism of bigraded vector spaces $\z(\u_1((\mathfrak {so}_7))) \cong \z(\u_2((\mathfrak {so}_7)))$ is reminiscent of a similar isomorphism found in \cite{LQ2}, where two non-conjugated blocks had isomorphic centers. This suggests that the corresponding categories might be Morita equivalent.

\subsection{Type \texorpdfstring{$A_4$}{A4}}

We check the conjectures \ref{conj:diag} and \ref{conj:triv} for $\g = \sl_5$: 

\begin{theorem}
There is an isomorphism of bigraded vector space $\z(\u_0(\sl_5)) \cong \z(\u_0(\sl_5)^{\g}  \cong {\mathrm{DC}}_5$. The table is as follows: \[ \begin{array}{r|l l l l l l l l l l l}
	{\scriptstyle i+j=0}&1&&&&&&&&&&\\
	{\scriptstyle i+j=2}&4&1&&&&&&&&&\\
	{\scriptstyle i+j=4}&9&5&1&&&&&&&&\\
	{\scriptstyle i+j=6}&15&14&5&1&&&&&&&\\
	{\scriptstyle i+j=8}&20&29&15&5&1&&&&&&\\
	{\scriptstyle i+j=10}&22&44&33&15&5&1&&&&&\\
	{\scriptstyle i+j=12}&20&51&54&34&15&5&1&&&&\\
	{\scriptstyle i+j=14}&15&46&66&58&34&15&5&1&&&\\
	{\scriptstyle i+j=16}&9&31&56&66&54&33&15&5&1&&\\
	{\scriptstyle i+j=18}&4&15&31&46&51&44&29&14&5&1&\\
	{\scriptstyle i+j=20}&1&4&9&15&20&22&20&15&9&4&1\\
	\hline h^{i,j}&{\scriptstyle j-i=0}&{\scriptstyle j-i=2}&{\scriptstyle j-i=4}&{\scriptstyle j-i=6}&{\scriptstyle j-i=8}&{\scriptstyle j-i=10}&{\scriptstyle j-i=12}&{\scriptstyle j-i=14}&{\scriptstyle j-i=16}&{\scriptstyle j-i=18}&{\scriptstyle j-i=20}
\end{array}
 \] 

\end{theorem}

Let us present the other bigraded tables. We label the singular blocks as before, i.e as a subset $J \subset \{1,2,3,4\}$. However we consider them up to the involution $i \mapsto 5-i$ (i.e up to conjugay in the extended affine Weyl group). There are also non-conjugated blocks giving isomorphic centers: $\z(\u_1(\sl_5)) \cong \z(\u_2(\sl_5))$, $\z(\u_{12}(\sl_5)) \cong z(\u_{23}(\sl_5))$ and $\z(\u_{13}(\sl_5)) \cong \z(\u_{14}(\sl_5))$. We also know that the block $\u_{123}(\sl_5)$ correspond to a projective space and was computed in \cite{LQ2}. Hence we only need to present $\u_1(\sl_5), \u_{12}(\sl_5), \u_{13}(\sl_5)$ and $\u_{124}(\sl_5)$. We will only present their dimension table, since for all blocks, the center is $\g$-invariant.

\begin{theorem}
The center of $\u_1(\sl_5)$ is as follows: \[ \begin{array}{r|l l l l l l l l l l}
	{\scriptstyle i+j=0}&1&&&&&&&&&\\
	{\scriptstyle i+j=2}&3&1&&&&&&&&\\
	{\scriptstyle i+j=4}&6&4&1&&&&&&&\\
	{\scriptstyle i+j=6}&9&10&4&1&&&&&&\\
	{\scriptstyle i+j=8}&11&18&11&4&1&&&&&\\
	{\scriptstyle i+j=10}&11&23&21&11&4&1&&&&\\
	{\scriptstyle i+j=12}&9&23&29&22&11&4&1&&&\\
	{\scriptstyle i+j=14}&6&17&28&29&21&11&4&1&&\\
	{\scriptstyle i+j=16}&3&9&17&23&23&18&10&4&1&\\
	{\scriptstyle i+j=18}&1&3&6&9&11&11&9&6&3&1\\
	\hline h^{i,j}&{\scriptstyle j-i=0}&{\scriptstyle j-i=2}&{\scriptstyle j-i=4}&{\scriptstyle j-i=6}&{\scriptstyle j-i=8}&{\scriptstyle j-i=10}&{\scriptstyle j-i=12}&{\scriptstyle j-i=14}&{\scriptstyle j-i=16}&{\scriptstyle j-i=18}
\end{array} \] 
The center of $\u_{12}(\sl_5)$ is as follows: \[ \begin{array}{r|l l l l l l l l}
	{\scriptstyle i+j=0}&1&&&&&&&\\
	{\scriptstyle i+j=2}&2&1&&&&&&\\
	{\scriptstyle i+j=4}&3&3&1&&&&&\\
	{\scriptstyle i+j=6}&4&5&3&1&&&&\\
	{\scriptstyle i+j=8}&4&7&6&3&1&&&\\
	{\scriptstyle i+j=10}&3&6&8&6&3&1&&\\
	{\scriptstyle i+j=12}&2&4&6&7&5&3&1&\\
	{\scriptstyle i+j=14}&1&2&3&4&4&3&2&1\\
	\hline h^{i,j}&{\scriptstyle j-i=0}&{\scriptstyle j-i=2}&{\scriptstyle j-i=4}&{\scriptstyle j-i=6}&{\scriptstyle j-i=8}&{\scriptstyle j-i=10}&{\scriptstyle j-i=12}&{\scriptstyle j-i=14}
\end{array} \]
The center of $\u_{13}(\sl_5)$ is as follows: \[ \begin{array}{r|l l l l l l l l l}
	{\scriptstyle i+j=0}&1&&&&&&&&\\
	{\scriptstyle i+j=2}&2&1&&&&&&&\\
	{\scriptstyle i+j=4}&4&3&1&&&&&&\\
	{\scriptstyle i+j=6}&5&7&3&1&&&&&\\
	{\scriptstyle i+j=8}&6&10&8&3&1&&&&\\
	{\scriptstyle i+j=10}&5&11&12&8&3&1&&&\\
	{\scriptstyle i+j=12}&4&9&14&12&8&3&1&&\\
	{\scriptstyle i+j=14}&2&5&9&11&10&7&3&1&\\
	{\scriptstyle i+j=16}&1&2&4&5&6&5&4&2&1\\
	\hline h^{i,j}&{\scriptstyle j-i=0}&{\scriptstyle j-i=2}&{\scriptstyle j-i=4}&{\scriptstyle j-i=6}&{\scriptstyle j-i=8}&{\scriptstyle j-i=10}&{\scriptstyle j-i=12}&{\scriptstyle j-i=14}&{\scriptstyle j-i=16}
\end{array} \]
The center of $\u_{124}(\sl_5)$ is as follows: 
\[ \begin{array}{r|l l l l l l l}
	{\scriptstyle i+j=0}&1&&&&&&\\
	{\scriptstyle i+j=2}&1&1&&&&&\\
	{\scriptstyle i+j=4}&2&2&1&&&&\\
	{\scriptstyle i+j=6}&2&3&2&1&&&\\
	{\scriptstyle i+j=8}&2&3&4&2&1&&\\
	{\scriptstyle i+j=10}&1&2&3&3&2&1&\\
	{\scriptstyle i+j=12}&1&1&2&2&2&1&1\\
	\hline h^{i,j}&{\scriptstyle j-i=0}&{\scriptstyle j-i=2}&{\scriptstyle j-i=4}&{\scriptstyle j-i=6}&{\scriptstyle j-i=8}&{\scriptstyle j-i=10}&{\scriptstyle j-i=12}
\end{array} \]

\end{theorem}

\section{Higher Hochschild cohomology groups}

We present some tables of higher Hochschild cohomology groups, obtained by our computer algorithm. These results are not complete, and there are additional results available on the repository of the algorithm: \url{https://github.com/RikVoorhaar/bgg-cohomology}. 

Recall that $\HH^{\bullet}(\u)$ is a Gerstenhaber algebra, and that the bracket is of degree $-1$. In particular, $\HH^1(\u_q(\g))$ acts on $\HH^0(\u_q(\g))$. This gives a large group of symmetries, and understanding it could lead to a better understanding of the center.

\begin{proposition}

For type $A_2$, for $4\leq s\leq 14$ and $s=2k$ even we get the following bigraded table
\[
\begin{array}{r|l l l}{\scriptstyle i+j=2}&L\left(k,k\right)^{2}&&\\
{\scriptstyle i+j=4}&0&
L(k,k-1)^2\dirsum L(k-1,k)^2\dirsum L(k,k)^4\dirsum L(k+1,k)^2\dirsum L(k,k+1)^2\\{\scriptstyle i+j=6}&0&0&
L(k,k)^2\\\hline h^{i,j}&{\scriptstyle j-i=0}&{\scriptstyle j-i=2}&{\scriptstyle j-i=4}\end{array}
\]
for $3\leq s\leq 15$ and $s=2k+1$ odd (the second column of the table can be deduced from the $\sl_2$ symmetry)

\[
\begin{array}{r|l l}{\scriptstyle i+j=3}&L(k,k)^2\dirsum L(k+1,k)^2\dirsum L(k,k+1)^2\dirsum L(k+1,k+1)\\ 
{\scriptstyle i+j=5}& 0
\\\hline h^{i,j}&{\scriptstyle j-i=1}\end{array}
\]
\end{proposition}

%In type $B_2$, based on $0\leq s\leq 6$ we conjecture that for $s\geq 4$:

%\[
%\begin{array}{r|l l l l}{\scriptstyle i+j=2}&\bigoplus_{\ell=0}^k L_{k,k+\ell}^{1+(k+\ell)\operatorname{mod} 2}\\
%{\scriptstyle i+j=4}& L_{k,2k}& ???\\
%{\scriptstyle i+j=6}&0&L_{k,2k-1}\dirsum L_{k,2k}^2\dirsum L_{k+1,2k+1} \\
%{\scriptstyle i+j=8}&0\\
%\hline h^{i,j}&{\scriptstyle i-j=0}&{\scriptstyle i-j=2}&{\scriptstyle i-j=4}&{\scriptstyle i-j=6}

%\end{array}
%\]

In type $G_2$, $\HH^1(\u_0)$ is given by the following table: 

\resizebox{0.85\textwidth}{!}{\setlength{\mathindent}{0pt}
\[
\begin{array}{r|l l l l l l}
	{\scriptstyle i+j=1}&\mathbb{C} \dirsum L_{3,2}&&&&&\\
	{\scriptstyle i+j=3}&\mathbb{C}^{2} \dirsum L_{2,1} \dirsum L_{3,2}^{2}&\mathbb{C} \dirsum L_{3,2}&&&&\\
	{\scriptstyle i+j=5}&\mathbb{C}^{2} \dirsum L_{2,1}^{2} \dirsum L_{3,2}^{2}&\mathbb{C}^{2} \dirsum L_{2,1} \dirsum L_{3,2}^{2}&\mathbb{C} \dirsum L_{3,2}&&&\\
	{\scriptstyle i+j=7}&\mathbb{C}^{2} \dirsum L_{2,1}^{2} \dirsum L_{3,2}&\mathbb{C}^{3} \dirsum L_{2,1}^{6} \dirsum L_{3,2}^{2} \dirsum L_{4,2}&\mathbb{C}^{2} \dirsum L_{2,1} \dirsum L_{3,2}^{2}&\mathbb{C} \dirsum L_{3,2}&&\\
	{\scriptstyle i+j=9}&\mathbb{C}^{3} \dirsum L_{2,1}&\mathbb{C}^{4} \dirsum L_{2,1}^{6} \dirsum L_{3,2} \dirsum L_{4,2}^{2}&\mathbb{C}^{3} \dirsum L_{2,1}^{6} \dirsum L_{3,2}^{2} \dirsum L_{4,2}&\mathbb{C}^{2} \dirsum L_{2,1} \dirsum L_{3,2}^{2}&\mathbb{C} \dirsum L_{3,2}&\\
	{\scriptstyle i+j=11}&\mathbb{C}^{2}&\mathbb{C}^{3} \dirsum L_{2,1}&\mathbb{C}^{2} \dirsum L_{2,1}^{2} \dirsum L_{3,2}&\mathbb{C}^{2} \dirsum L_{2,1}^{2} \dirsum L_{3,2}^{2}&\mathbb{C}^{2} \dirsum L_{2,1} \dirsum L_{3,2}^{2}&\mathbb{C} \dirsum L_{3,2}\\
	\hline h^{i,j}&{\scriptstyle j-i=1}&{\scriptstyle j-i=3}&{\scriptstyle j-i=5}&{\scriptstyle j-i=7}&{\scriptstyle j-i=9}&{\scriptstyle j-i=11}
\end{array}
\]
}

As noticed in \cite{LQ3} already, we obtain that $\g$ is a subalgebra of $\HH^1$, sitting in bidegree $(1,1)$. Their results imply that this copy of $\g$ is acting in a compatible way on the center with both the algebraic and geometric existing $\g$-action. We hope to understand the full action of $\HH^1$ on $\HH^0$, by understanding the geometric action of $\HH^1_{\C^*}(\NN)$ on $\HH^0_{\C^*}(\NN)$. \\

We also computed $\HH^s(\u_q(\g))$ for $s \leq 3$ and $\g = \sl_4$. We present the tables for $s=1$: 
\begin{proposition}

The group $\HH^1(\u_q(\sl_4))$ is given by the following table: 
\setlength{\mathindent}{0pt}
\[ \begin{array}{r|l l l l l l}
	{\scriptstyle i+j=1}&\mathbb{C}\dirsum L_{1,1,1}&&&&&\\
	{\scriptstyle i+j=3}&\mathbb{C}^{4}\dirsum L_{1,1,1}^{3}&\mathbb{C}\dirsum L_{1,1,1}&&&&\\
	{\scriptstyle i+j=5}&\mathbb{C}^{9}\dirsum L_{1,1,1}^{5}&\mathbb{C}^{5}\dirsum L_{1,1,1}^{4}&\mathbb{C}\dirsum L_{1,1,1}&&&\\
	{\scriptstyle i+j=7}&\mathbb{C}^{11}\dirsum L_{1,1,1}^{3}&\mathbb{C}^{13}\dirsum L_{1,1,1}^{8}\dirsum L_{1,2,1}^{3}&\mathbb{C}^{5}\dirsum L_{1,1,1}^{4}\dirsum L_{1,2,1}&\mathbb{C}\dirsum L_{1,1,1}&&\\
	{\scriptstyle i+j=9}&\mathbb{C}^{8}&\mathbb{C}^{17}\dirsum L_{1,1,1}^{5}\dirsum L_{1,2,1}^{3}&\mathbb{C}^{13}\dirsum L_{1,1,1}^{8}\dirsum L_{1,2,1}^{3}&\mathbb{C}^{5}\dirsum L_{1,1,1}^{4}&\mathbb{C}\dirsum L_{1,1,1}&\\
	{\scriptstyle i+j=11}&\mathbb{C}^{3}&\mathbb{C}^{8}&\mathbb{C}^{11}\dirsum L_{1,1,1}^{3}&\mathbb{C}^{9}\dirsum L_{1,1,1}^{5}&\mathbb{C}^{4}\dirsum L_{1,1,1}^{3}&\mathbb{C}\dirsum L_{1,1,1}\\
	\hline h^{i,j}&{\scriptstyle j-i=1}&{\scriptstyle j-i=3}&{\scriptstyle j-i=5}&{\scriptstyle j-i=7}&{\scriptstyle j-i=9}&{\scriptstyle j-i=11}
\end{array} \]
\end{proposition}

\subsection{Limits of the algorithm}

With the computational resources available to the authors, it is not feasible to do the computation of the center of the principal block for type $D_4$ and $B_4$ with the current implementation of the algorithm. Hence, it seems that most of the computations are not accessible in rank $\geq 4$. For higher Hochschild cohomology, in type $A_2$ we could go up to $s \leq 13$, but for type $B_2$ only $s \leq 7$. However, looking at singular blocks makes the computation much easier, and even for $A_4$ we could compute the Hochschild cohomology of the projective spaces up to $s \leq 6$ (this specific computation will be detailed in the next section). 

\section{Remarks on projective spaces}

We now look at $X = \P^n = G/P$. For $s=0$, the Hochschild cohomology was computed in \cite{LQ2}. Let us call a bidegree $(i,j)$ \emph{positive} if $i>0$, and let us call the positive part of $\HH^s_{\C^*}(T^*(G/P))$ the direct sum of the corresponding bigraded summands. For $s=1$, we show that the positive Hochschild cohomology is generated by the first column under the $\sl_2$-action. For $s \geq 2$, we prove a vanishing property, and we compute several Hochschild cohomology groups for $T^*\P^3$ and $T^*\P^4$. 

\subsection{A subalgebra of \texorpdfstring{$\HH^1$}{HH1}}

In this subsection we will compute the positive part of $\HH^1_{\C^*}(\P^n)$, and we present the full computation for $3 \leq n \leq 6$. We will need the following lemma:
\begin{lemma}[\cite{LQ2}]\label{lem:cohomologpn}
Let $p' \neq p \leq n$. Then $\H^{p'}(\P^n, \Omega_{\P^n}^{\tensor p}) = 0$. Moreover, $\H^p(\P^n, \Omega_{\P^n}^{\tensor p}) = \C$.  
\end{lemma}

Let us emphasize that the non-trivial cohomology comes from the split summand $\wedge^p \Omega_{\P^n} \subset \Omega^{\otimes p}$. We now want to compute the first column of $\HH^1_{\C^*}(\P^n)$. For this, we introduce $1 \leq p \leq n$ such that $j=i+1 = p$. It follows that $k = -2p$, so $s = i+j+k = 1$.

\begin{proposition}\label{prop:firstcolumn}
Let $1 \leq p \leq n$.  Then $\H^{p-1}(\P^n, \wedge^p \pr_*(T_{\NN_P})^{-2p}) = \C$ if $p \neq 1$ and $\C \dirsum \g$ if $p = 1$. 
\end{proposition}

\begin{proof}
First, we assume $p >1$. By definition, there is a short exact sequence
\[
\begin{tikzcd}[column sep=1.5em]
    0 \ar[r]& T_{\P^n} \tensor \wedge^{p} \Omega_{\P^n} \ar[r]& \wedge^{p} \pr_*(T_{\NN_P})^{-2p} \ar[r]&     T_{\P^n} \tensor \wedge^{p+1} \Omega_{\P^n} \ar[r]& 0 
\end{tikzcd}
\]

As usual, let $\omega_{\P^n} = \det(\Omega_{\P^n})$ be the canonical bundle on $\P^n$. Using $\wedge^p \Omega_{\P^n} \cong \wedge^{n-p} T_{\P^n} \tensor \omega_{\P^n}$, we obtain by Serre duality that
\[
    \H^p(\P^n, T_{\P^n} \tensor \wedge^p \Omega_{\P^n}) \cong \H^{n-p}(\Omega_{\P^n} \tensor \wedge^{n-p} \Omega_{\P^n})^*.
\] 

By lemma \ref{lem:cohomologpn}, we know that the cohomology group $\H^q(\P^n, \Omega_{\P^n} \otimes \wedge^m \Omega_{\P^n})$ is nonzero if and only if ${q = m+1}$. Hence we have that
\[
    \H^{p-1}(\P^n, T_{\P^n} \tensor \wedge^p \Omega_{\P^n}) = \C,
    \quad\H^{p-1}(\P^n, T_{\P^n} \tensor \wedge^{p+1} \Omega_{\P^n}) = 0,
    \quad \text{and}\quad 
    \H^{p-2}(\P^n, T_{\P^n} \tensor \wedge^{p+1} \Omega_{\P^n}) = 0.
\]
It follows that the long exact sequence in cohomology gives the short exact sequence 
\[
\begin{tikzcd}[column sep=1.5em]
    0 \ar[r]& \C \ar[r]& \H^{p-1}(\P^n, \wedge^p \pr_*(T_{\NN_P})^{-2p}) \ar[r]& 0
\end{tikzcd}
\]
proving the proposition if $p>1$. 

If $p= 1$, then the composition factors of $(\pr_* T_{\NN_P})^0$ are $T_{\P^n} \tensor \Omega_{\P^n}$ and $T_{\P^n}$. We have \[\H^0(\P^n,T_{\P^n} \tensor \Omega_{\P^n}) = \C,\quad \text{respectively}\quad \H^0(\P^n, T_{\P^n}) = \g.\] By the Borel-Bott-Weil theorem, $\H^1(\P^n, T_{\P^n}) = 0$ so the result follows using the long exact sequence. 
\end{proof}

\begin{corollary}\label{subalgebra}
There is a subalgebra of $\HH^1_{\C^*}(T\P^n)$ of dimension ${(n+1)n}/{2} + (n^2-1)n$.
\end{corollary}

\begin{proof}
It follows easily from the existence of the $\sl_2$-action on $\HH^1_{\C^*}(\P^n)$, and the fact that the Poisson bivector field has bidegree $(i,j) = (0,2)$. 
\end{proof}

\begin{proposition}
The positive part of $\HH^1_{\C^*}(\P^n)$ coincides with the positive part of the subalgebra from the previous corollary, i.e for $(i,j)$ with $i \geq 1$ we have $\HH^1_{\C^*}(\P^n)^{(i,j)} = \C$. 
\end{proposition}

\begin{proof}
We fix $1 \leq i \leq n \leq j \leq 2n$, where $i+j=1 \mod 2$ and let $t = j-i$, and $r = (t+1)/2$.  The corresponding cohomology group is $\HH^1_{\C^*}(\P^n)^{(i,j)} = H^i(\NN_P, \bigwedge^{i+t} T \NN_P)^{-2i-t+1}$. The sheaf $\pr_*(\bigwedge^{i+t} T \NN_P)^{-2i-t+1}$ has a filtration by factors $\mathcal F_l$ (where $0 \leq l \leq r$) given by \[ \mathcal F_l := S^{r-l} T_{\P^n} \otimes \bigwedge^{l} T_{\P^n} \otimes \bigwedge^{i+t-l} \Omega_{\P^n} \] 
We want to compute $\H^i(\mathcal F_l)$. Using Serre duality we obtain an isomorphism 
\[ 
    \H^i(\mathcal F_l) \cong \H^{n-i} \left(S^{r-l} \Omega_{\P^n} \otimes \bigwedge\nolimits^{\!l} \Omega_{\P^n} \otimes \bigwedge\nolimits^{\!n+l-i-t} \Omega_{\P^n}\right)^*
\]
Denote the last factor in the tensor product above by $\mathcal Q_l$. Note that $\mathcal Q_l$ is then a direct summand of $ \Omega^{\otimes(n+l+1-i-r)}$. Notice that $l+1-r-i \leq 0$, because $l \leq r$ and $i \geq 1$ by hypothesis, hence we can apply lemma \ref{lem:cohomologpn}. When $l=r-1$, we have 
\[
    \mathcal Q_l = \Omega_{\P^n} \otimes \bigwedge\nolimits^{\!r-1} \Omega_{\P^n} \otimes \bigwedge\nolimits^{\!n-i-r} \Omega_{\P^n},
\]
hence $\mathcal Q_l$ contains the direct summand $\wedge^{n-i} \Omega_{\P^n}$, so by the lemma \ref{lem:cohomologpn} we obtain $\H^{n-i}(\mathcal F_{r-1}) = \C$. For other values of $l$, $\mathcal Q_l$ does not contribute. Therefore, again by lemma \ref{lem:cohomologpn}, we have $\H^{n-i}(\mathcal F_{l}) = 0$. This shows that $\dim \HH^1_{\C^*}(\P^n)^{(i,j)} \leq 1 $, and using the $\sl_2$-action and proposition \ref{prop:firstcolumn} we conclude that $\dim \HH^1_{\C^*}(\P^n)^{(i,j)} = 1$. 
\end{proof}

One can visualize this result as follows (we let $\g' := \C \oplus \g$, and \textbf{?} denotes an unknown factor): 
\[ \begin{array}{c|c c c c c c c}
	{\scriptstyle i+j=1}&\g' &&&&&&\\
	{\scriptstyle i+j=3}&\mathbb{C}& \g' \dirsum \,\textbf{?} &&&&&\\
	{\scriptstyle i+j=5}&\mathbb{C}&\mathbb{C}& \g' \dirsum \,\textbf{?}&&&&\\
	{\scriptstyle i+j=7}&\mathbb{C}&\mathbb{C}&\mathbb{C}&\g' \dirsum \,\textbf{?}&&&\\
	{\scriptstyle \vdots}&\mathbb{C}&\mathbb{C}&\mathbb{C}&\mathbb{C}& \dots & \g' \dirsum \,\textbf{?}&\\
	{\scriptstyle i+j=2n-1}&\mathbb{C}&\mathbb{C}&\mathbb{C}&\mathbb{C}& \dots & \C & \g' \\
	\hline h^{i,j}&{\scriptstyle j-i=1}&{\scriptstyle j-i=3}&{\scriptstyle j-i=5}&{\scriptstyle j-i=7}&{\dots}&{\scriptstyle j-i=2n-3}&{\scriptstyle j-i=2n-1}
\end{array} \]
We computed the bigraded summands of $\HH^1_{\C^*}(\P^n)$ for $i>0$ in the previous proposition. We cannot compute these summands for $i=0$ in general, but we present the missing bigraded summand for $3 \leq n \leq 6$ (giving the complete $\HH^1$ using the $\sl_2$-symmetry), obtained using our computer algorithm:
\begin{proposition}
For $3 \leq n \leq 6$, the bigraded summands $\HH^1_{\C^*}(\P^n)^{(0,3)}$ and $\HH^1_{\C^*}(\P^n)^{(0,5)}$ are given by the following table: 
\[ \begin{array}{r|ll}
     {\scriptstyle n=3} & \g' \dirsum L_{1,2,1} & \g' \\ 
	 {\scriptstyle n=4}& \g' \dirsum L_{1,2,2,1} & \g' \\
	 {\scriptstyle n=5}& \g' \dirsum L_{1,2,2,2,1} &\g' \dirsum  L_{1,2,2,2,1} \dirsum L_{1,2,3,2,1} \\ 
	 {\scriptstyle n=6}& \g' \dirsum L_{1,2,2,2,1} & 	\g' \dirsum L_{1,2,2,2,1} \dirsum L_{1,2,3,2,1} \\  
    \hline  {\scriptstyle i=0} & {\scriptstyle j=3} & {\scriptstyle j=5}
	\\  
\end{array} \] 
\end{proposition}

\subsection{Higher Hochschild cohomology groups}

We prove that $\HH^s_{\C^*}(\P^n)^{(i,j)} = 0$ if $s \geq 5$ and $i \geq s-2$. We also present some low-degree Hochschild cohomology groups for $n=3,4$. We begin by a lemma:

\begin{lemma}\label{lem:wedge}
Let $a,b,c$ be positive integers where $a \geq 3$ and $a+b+c = j$. Then, if $V$ is a finite-dimensional vector space, the $\SL(V)$-module $\Sym^a(V) \otimes \wedge^b(V) \otimes \wedge^c(V)$ does not contain a copy of $\wedge^j(V)$. 
\end{lemma}

\begin{proof}
Let $V = \C^{n+1}$, and $\Gamma_{a_1,\dots, a_n}$ be the irreducible $\SL(V)$-representation of highest weight $\sum_i a_i\varpi_i$, where $\varpi_i$ are the fundamental weights. We have $\Gamma_{a,0,0, \dots} \cong \Sym^a(V)$ and $\Gamma_{0,0,\dots,0,1,0,\dots,0} \cong \wedge^b(V)$ (where the $1$ is at the $b$-th position). We recall that as a special case of the Littlewood-Richardson rule we have, 
\[
    \Gamma_{a_1, \dots, a_n} \otimes \wedge^k V \cong \bigoplus_{(b_1, \dots, b_n) \in \mathcal B} \Gamma_{b_1, \dots, b_n}, 
\] 
where $\mathcal B$ is a certain subset of $\mathbb N^n$, such that if $(b_1, \dots, b_n) \in \mathcal B$ then $|b_i-a_i| \leq 1$. We get \[ \Sym^a(V) \otimes \wedge^b(V) \otimes \wedge^c(V) \cong \bigoplus \Gamma_{c_1, \dots, c_n}\] where $(c_1, \dots, c_n)$ are certain integers such that $c_1 \geq 1$ so the lemma follows. 
\end{proof}

\begin{proposition}
If $s \geq 4$, and $s-2 \leq i$, we have $\HH^s_{\C^*}(\P^n)^{(i,j)} = 0$. 
\end{proposition}

\begin{proof}
As before, we fix $1 \leq i \leq n \leq j \leq 2n$, and $1 \leq r \neq n$ is the column in the bigraded table. We want to compute $\H^i(\NN_P, (\wedge^{s+i+2(r-1)}T \NN_P)^{-2i-2r+2}$. As usual $\pr_* \wedge^{s+i+2(r-1)}T \NN_P)^{-2i-2r+2}$ has a filtration with summands 
\[
\mathcal F_l := S^{r-l+s-2} T_{\P^n} \otimes \bigwedge\nolimits^{\!l} T_{\P^n} \otimes \bigwedge\nolimits^{\!s+i+2(r-1)-l} \Omega_{\P^n} 
\] 

Again, by Serre duality we have
\[ 
\H^i(\P^n, \mathcal F_l) \cong \H^{n-i}(\P^n, S^{r-l+s-2} \Omega_{\P^n} \otimes \bigwedge\nolimits^{\!l} \Omega_{\P^n} \otimes \bigwedge\nolimits^{\!n-s-i-2(r-1)+l} \Omega_{\P^n})^*
\] 

Denote the last sheaf by $\mathcal Q_l$, it is a direct summand of $\Omega^{\otimes(n-i+l-r)}$.  We know $l \leq r +s - 2$ and by hypothesis $s-2 \leq i$ hence $l \leq i+r$. So we can apply lemma \ref{lem:cohomologpn} to $p= n - i + l - r$, and get $\H^i(\P^n, \mathcal F_l) = 0$ if $l \neq r$. If $l=r$, the only summand that can contribute to the cohomology is $\wedge^p \Omega$. But $\mathcal F_r$ does not contains $\wedge^r \Omega_{\P^n}$ by lemma \ref{lem:wedge}, so we have $\H^i(\P^n, \mathcal F_l) = 0$ as well. 
\end{proof}

We think that this vanishing holds more generally for all $i>0$ and $s \geq 2$, based on our computer computations. This should hold only for projective spaces: already for Grassmannians we found several counterexamples. 

We conclude this section by presenting table of  $\HH^s_{\C^*}(\P^3)$ for $0 \leq s \leq 9$ and $\HH^s_{\C^*}(\P^3)$ for $0 \leq s \leq 6$. We present truncated tables for readability, but the remaining part can be deduced from the $\sl_2$-action, as explained in section $2$. 

\begin{proposition}

The group $\HH^6_{\C^*}(\P^3)$ is given by the following table: 
\setlength{\mathindent}{0pt}
\[ \begin{array}{r|l l l l}
	{\scriptstyle i+j=0}&L_{3,3,3}&\\
	{\scriptstyle i+j=2}&0&L_{3,3,2}L_{2,3,3}L_{4,3,2}L_{3,3,3}^{2}L_{2,3,4}L_{3,4,3}L_{4,4,3}L_{3,4,4}\\
	{\scriptstyle i+j=4}&0&0\\
	{\scriptstyle i+j=6}&0&0\\
	\hline h^{i,j}&{\scriptstyle j-i=0}&{\scriptstyle j-i=2}
\end{array} \]

The group $\HH^7_{\C^*}(\P^3)$ is given by the following table: 
\[\begin{array}{r|l l l}
	{\scriptstyle i+j=1}&L_{3,3,3}L_{4,4,3}L_{3,4,4}L_{4,4,4}&\\
	{\scriptstyle i+j=3}&0&L_{4,3,2}L_{3,3,3}L_{2,3,4}L_{3,4,3}L_{4,4,3}^{2}L_{3,4,4}^{2}L_{5,4,3}L_{4,4,4}L_{3,4,5}L_{4,5,4}\\
	{\scriptstyle i+j=5}&0&0\\
	\hline h^{i,j}&{\scriptstyle j-i=1}&{\scriptstyle j-i=3}
\end{array}\]

The group $\HH^2_{\C^*}(\P^4)$ is given by the following table:
\[ \begin{array}{r|l l l l l}
	{\scriptstyle i+j=0}&L_{1,1,1,1}&&\\
	{\scriptstyle i+j=2}&0&L_{1,1,1,1}^{2}L_{1,2,2,1}L_{2,2,2,1}L_{1,2,2,2}&\\
	{\scriptstyle i+j=4}&0&0&L_{1,1,1,1}^{2}L_{1,2,2,1}^{2}L_{2,2,2,1}L_{1,2,2,2}L_{2,3,2,1}L_{1,2,3,2}\\
	{\scriptstyle i+j=6}&0&0&0\\
	{\scriptstyle i+j=8}&0&0&0\\
	\hline h^{i,j}&{\scriptstyle j-i=0}&{\scriptstyle j-i=2}&{\scriptstyle j-i=4}
\end{array} \]

The group $\HH^3_{\C^*}(\P^4)$ is given by the following table:

\resizebox{0.95\textwidth}{!}{
\[ \begin{array}{r|l l l l}
	{\scriptstyle i+j=1}&L_{1,1,1,1}L_{2,2,2,1}L_{1,2,2,2}L_{2,2,2,2}&\\
	{\scriptstyle i+j=3}&0&L_{1,1,1,1}L_{1,2,2,1}L_{2,2,2,1}^{2}L_{1,2,2,2}^{2}L_{2,3,2,1}L_{2,2,2,2}L_{1,2,3,2}L_{3,3,2,1}L_{1,2,3,3}L_{2,3,3,2}\\
	{\scriptstyle i+j=5}&0&0\\
	{\scriptstyle i+j=7}&0&0\\
	\hline h^{i,j}&{\scriptstyle j-i=1}&{\scriptstyle j-i=3}
\end{array} \] }

\end{proposition}

\printbibliography 

\end{document}